\documentclass[reqno,12pt]{amsart} 

\usepackage{amsmath,amssymb,amsthm,enumerate,mathrsfs,colonequals,bm}
\usepackage{fullpage,url,tikz,xspace,setspace, graphicx}
\usepackage{microtype}
\usepackage{calc}

\usepackage[
unicode=true
]{hyperref}
\usepackage[alphabetic,lite,nobysame]{amsrefs} 
\usepackage{color}
\newcommand{\bruce}[1]{\textsf{\color{blue} $\spadesuit$ Bruce: #1}}

\let\oldtocsection=\tocsection
\let\oldtocsubsection=\tocsubsection
\let\oldtocsubsubsection=\tocsubsubsection
\renewcommand{\tocsection}[2]{\hspace{0em}\oldtocsection{#1}{#2}}
\renewcommand{\tocsubsection}[2]{\hspace{1em}\oldtocsubsection{#1}{#2}}
\renewcommand{\tocsubsubsection}[2]{\hspace{2em}\oldtocsubsubsection{#1}{#2}}

\newcommand{\Bmu}{\mbox{$\raisebox{-0.59ex}
  {$l$}\hspace{-0.18em}\mu\hspace{-0.88em}\raisebox{-0.98ex}{\scalebox{2}
  {$\color{white}.$}}\hspace{-0.416em}\raisebox{+0.88ex}
  {$\color{white}.$}\hspace{0.46em}$}{}}

\newtheorem{theorem}{Theorem}
\newtheorem{proposition}[theorem]{Proposition}
\newtheorem{lemma}[theorem]{Lemma}
\newtheorem{corollary}[theorem]{Corollary}
\newtheorem{definition1}[theorem]{Definition}
\newtheorem{remark1}[theorem]{Remark}
\newcommand{\Z}{\mathbb{Z}}
\newcommand{\NN}{\mathbb{N}}
\newcommand{\Q}{\mathbb{Q}}
\newcommand{\BB}{\mathbb{B}}
\newcommand{\cM}{\mathcal{M}}
\newcommand{\wgr}{\widetilde{\mathit{gr}}}
\newcommand{\F}{\mathbb{F}}

\newcommand{\ggr}{{\sf{Gr}}}
\newcommand{\bFp}{\overline{\F}_p}
\newcommand{\oFp}{\overline{\F}_p}
\newcommand{\R}{\mathbb{R}}
\renewcommand{\H}{\mathbb{H}}
\newcommand{\Hh}{{\widehat{\H}}}
\newcommand{\GR}{\mathit{Gr}}
\newcommand{\BR}{\mathit{Br}}
\newcommand{\A}{\mathscr{A}}
\newcommand{\X}{\mathscr{X}}
\newcommand{\Aa}{\mathbb{A}}

\newcommand{\cSS}{\mathcal{S}}
\newcommand{\cP}{\mathscr{P}}
\newcommand{\msL}{\mathscr{L}}
\newcommand{\caP}{\mathcal{P}}
\newcommand{\cQ}{\mathcal{Q}}
\newcommand{\bs}{\backslash}
\newcommand{\bA}{\mathbb{A}}
\newcommand{\cO}{\mathcal{O}}
\newcommand{\OO}{\cO}
\newcommand{\cB}{\mathcal{B}}

\newcommand{\msH}{\mathscr{H}}
\newcommand{\omsH}{\overline{\msH}}
\newcommand{\msM}{\mathscr{M}}
\newcommand{\gr}{\mathit{gr}}
\newcommand{\br}{\mathit{br}}

\newcommand\sS{{\sf SP}}
\newcommand{\sSgpz}{\sS_{\!g}(p)_0}

\newcommand{\msP}{\mathscr{P}}
\newcommand{\HH}{\mathbb{H}}
\DeclareMathOperator\N{N}
\newcommand{\MM}{\operatorname{Mat}_{g\times g}}

\DeclareMathOperator\Hom{Hom}
\DeclareMathOperator\Gal{Gal}
\DeclareMathOperator\End{End}
\DeclareMathOperator\Id{Id}
\DeclareMathOperator\GSp{GSp}
\DeclareMathOperator\Sp{Sp}
\DeclareMathOperator\rU{U}
\DeclareMathOperator\Pic{Pic}
\DeclareMathOperator\HNm{HNm}
\DeclareMathOperator\PU{PU}
\DeclareMathOperator\PGU{PGU}
\DeclareMathOperator\GU{GU}
\DeclareMathOperator\SU{SU}
\DeclareMathOperator\SL{SL}
\DeclareMathOperator\SLT{\SL_2}
\DeclareMathOperator\GL{GL}
\DeclareMathOperator\PGL{PGL}
\DeclareMathOperator\Ta{Ta}
\DeclareMathOperator\GG{G}
\DeclareMathOperator\Gr{Gr}
\DeclareMathOperator\Ver{Ver}
\DeclareMathOperator\val{val}
\DeclareMathOperator\Ed{Ed}
\DeclareMathOperator\Iso{Iso}
\DeclareMathOperator\Disc{Disc}
\DeclareMathOperator\Mat{Mat}
\DeclareMathOperator\Ad{Ad}
\DeclareMathOperator\Pf{Pf}
\DeclareMathOperator\id{id}

\DeclareMathOperator\charr{char}
\DeclareMathOperator\Grr{Gr}

\DeclareMathOperator\Frobb{Frob}
\DeclareMathOperator\w{w}
\DeclareMathOperator\Adw{\Ad_w}
\DeclareMathOperator\iso{iso}
\DeclareMathOperator\Aut{Aut}
\DeclareMathOperator\Norm{Nm}
\DeclareMathOperator\rdeg{rdeg}
\DeclareMathOperator\Jac{Jac}

\begin{document}
\title{Isogeny graphs of superspecial abelian varieties
  and Brandt matrices}
\subjclass[2010]{Primary 14K02; Secondary 11G10, 14G15}
\keywords{superspecial, abelian varieties, isogeny graphs, Brandt matrices,
quaternionic unitary group}

\author{Bruce~W.~Jordan}
\address{Department of Mathematics, Baruch College, The City University
of New York, One Bernard Baruch Way, New York, NY 10010-5526, USA}
\email{bruce.jordan@baruch.cuny.edu}

\author{Yevgeny~Zaytman}
\address{Newton, MA 02465, USA}
\email{gzaytman@alum.mit.edu}

\begin{abstract}
Fix primes $p$ and $\ell$ with $\ell\neq p$.
If $(A,\lambda)$ is a $g$-dimensional principally polarized abelian variety,
an $(\ell)^g$-isogeny of $(A,\lambda)$ has kernel a maximal isotropic
subgroup of the $\ell$-torsion of $A$; the image has a natural
principal polarization.  
In this paper we study the isogeny graphs of
$(\ell)^g$-isogenies
of principally polarized superspecial abelian varieties in characteristic
$p$.
We define three isogeny graphs associated to such $(\ell)^g$-isogenies --
the big isogeny graph $\GR_{\!g}(\ell,p)$, the little isogeny graph
$\gr_{\!g}(\ell,p)$, and the 
enhanced isogeny graph $\wgr_{\!g}(\ell, p)$.
We apply strong approximation for the quaternionic unitary group to prove both that  
$\gr_{\!g}(\ell,p)$ and $\GR_{\!g}(\ell,p)$ are connected and that they are
not bipartite. The connectedness of the enhanced isogeny
graph $\wgr_{\!g}(\ell,p)$ then follows. 
The quaternionic unitary group has previously been applied to moduli
of abelian varieties in characteristic $p$ (sometimes invoking strong
approximation) by Chai, Ekedahl/Oort, and Chai/Oort.
The adjacency matrices
of the three isogeny graphs are given in terms of the Brandt matrices
defined by Hashimoto, Ibukiyama, Ihara, and Shimizu. We
study some basic properties of these Brandt matrices
and recast the theory using the notion of Brandt graphs.
We show that the isogeny graphs $\GR_{\!g}(\ell, p)$ and $\gr_{\!g}(\ell, p)$ are
in fact our Brandt graphs. 
We give the $\ell$-adic
uniformization of $\gr_{\!g}(\ell,p)$ and $\wgr_{\!g}(\ell,p)$.
The $(\ell+1)$-regular isogeny graph $\GR_1(\ell,p)$ 
for supersingular elliptic
curves is well known to be
Ramanujan.  We calculate the Brandt matrices for a range of $g>1$, $\ell$,
and $p$.  These calculations give four examples with $g>1$
where the regular graph $\GR_{\!g}(\ell,p)$ has two vertices
and is Ramanujan, and all other examples we computed  with $g>1$ and
two or more vertices were not Ramanujan.  In particular, the 
$(\ell)^g$-isogeny graph is not in general Ramanujan for $g>1$.
\end{abstract}
\maketitle
\vspace*{.14in}

\newpage

\renewcommand{\baselinestretch}{0.94}\normalsize
\tableofcontents
\renewcommand{\baselinestretch}{1.00}\normalsize

\newpage

\section{Introduction}
A superspecial abelian variety $A/\oFp$ of dimension $g$ is isomorphic
to a product of $g$ supersingular elliptic curves.
  If $g>1$, surprisingly
all such products are isomorphic to each other by Theorem \ref{sunn} below.
Fix a supersingular
elliptic curve $E/\oFp$ with $\cO=\cO_E=\End(E)$ a maximal order in
the rational definite quaternion algebra $\HH_p$ ramified at $p$.
\begin{theorem}
\label{sunn}
\textup{(Deligne, Ogus \cite{og}, Shioda \cite{shi}) }
Suppose $A/\oFp$ is a superspecial abelian variety with
$\dim A=g>1$.  Then $A\cong E^g$.
\end{theorem}
\noindent So for dimension $g=1$ there are {\em{many}} superspecial abelian
varieties ($=$ supersingular elliptic curves) each with {\em{one}}
principal polarization, but for $g>1$ there is {\em{one}} superspecial
abelian variety with {\em{many}} principal polarizations.

Let $\A=(A=E^g,\lambda)$ be a principally polarized superspecial abelian
variety of dimension $g$ over $\oFp$ with $\oFp$-isomorphism class
$[\A]$. The principal polarization $\lambda$  is an isomorphism
 from $A$ to $\hat{A}=\Pic^{0}(A)$
satisfying the conditions of Definition \ref{polar}.
The number $h=h_g(p)$  of such isomorphism classes $[\A]$ is finite
and is a type of class number.  For $g\geq 1$ set
\begin{equation}
\label{cloud}
\begin{split}
\sS_{\!g}(p)_0 & = \{\oFp\text{-isomorphism classes $[\A]$}\}\\
  & = \{ [\A_1],\ldots , [\A_h]\}\,\text{ with }\,\A_j=(A_j=E^g, \lambda_j)
\text{ if $g>1$.}
\end{split}
\end{equation}
So, for example,
\begin{align*}
\sS_1(p)_0 & =\{\text{supersingular $j$-invariants in characteristic $p$}\}
\text{ and}\\
\#\sS_{1}(p)_0 &  =h_1(p)=h(\HH_p), \text{ the class number of the quaternion
algebra $\HH_p$.}
\end{align*}

A principal polarization $\lambda$ on the abelian variety $A/\oFp$ defines
a Weil pairing on $A[n]$ with $(n,p)=1$: 
$\langle \,\,\, ,\,\,\,\rangle_{\lambda,n}\colon A[n]\times A[n]\rightarrow
\Bmu_{n}$. For $(n,p)=1$, put
\begin{equation}
\label{grid}
\Iso_n(\A)=\{\text{maximal isotropic subgroups $C\subseteq A[n]$}\}\quad
\text{with}\quad N_g(n)\colonequals \#\Iso_n(\A).
\end{equation}
Note that $N_g(n)$ is the number of maximal isotropic subgroups
of the standard nondegenerate symplectic $\Z/n\Z$-module of rank $2g$.
In case $n=\ell\neq p$ is prime we have
\begin{equation}
\label{grid1}
\#\Iso_\ell(\A)\equalscolon N_g(\ell)=\prod_{k=1}^g(\ell^k + 1);
\end{equation}
see, for example,  \cite[p.~419]{Pl}.
Suppose $C\subseteq A[\ell]$ is a subgroup with corresponding isogeny
$\psi_C:A\rightarrow A/C\equalscolon A'$.  Then there is a principal
polarization $\lambda'$ on $A'$ so that $\psi_C^\ast(\lambda')=\ell \lambda$ 
if and only if $C\in\Iso_\ell(\A)$.  In this
case write $\A'=(A',\lambda')=\A/C$ and say that $\psi_C$ is an
$(\ell)^g$-isogeny.  If $[\A]\in\sS_g(p)_0$, then
$[\A']\in\sS_g(p)_0$.  Such $(\ell)^g$-isogenies induce
correspondences from the finite set $\sS_{\!g}(p)_0$ to itself.  These
correspondences can be used to define various graphs---in this paper
we define {\em three} $(\ell)^g$-isogeny graphs: the {\em big} isogeny
graph $\GR_{\!g}(\ell,p)$, the {\em little} isogeny graph
$\gr_{\!g}(\ell, p)$, and the {\em enhanced} isogeny graph
$\wgr_{\!g}(\ell,p)$.  The literature seems 
to have only one 
isogeny graph; this ubiquitous graph is the big isogeny graph 
$\GR_{\!g}(\ell,p)$
for us. 

Distinguishing between these three 
makes many results clearer and more precise.  Take the case $g=1$ for 
example:
the little and enhanced isogeny graphs are uniformized by the 
Bruhat-Tits tree $\Delta=\Delta_{\ell}$ of $\SLT(\Q_{\ell})$;
the big isogeny graph $\GR_{\!1}(\ell,p)$ is not, cf. Section \ref{g=1q}.
And it is $\gr_{\!1}(\ell,p)$ and $\wgr_{\!1}(\ell,p)$
which arise from the bad reduction of Shimura curves and not
the familiar big isogeny graph $\GR_{\!1}(\ell,p)$ as we show in
Section \ref{eel}.
For general $g\geq 1$, the big isogeny graph
$\GR_{\!g}(\ell, p)$ is a regular graph by Theorem \ref{adbig}\eqref{adbig3},
so it is natural to ask if it is Ramanujan, whereas the little isogeny
graph $\gr_{\!g}(\ell,p)$ and the  
enhanced isogeny graph $\wgr_{\!g}(\ell,p)$ are not regular.

In this introduction we content ourselves with
defining the simplest of the three, the big isogeny graph $\GR=
\GR_{\!g}(\ell,p)$:
\begin{definition1}
\label{damp}
{\rm
The vertices of the graph $\GR=\GR_{\!g}(\ell,p)$ are
$\Ver(\GR)=\sS_{\!g}(p)_0$, so $h=h_g(p)=\#\Ver(\GR)$.
The (directed) edges of $\GR$ connecting the vertex
$[\A_i]\in\sS_{\!g}(p)_0$ to the vertex $[\A_j]\in\sS_{\!g}(p)_0$ are
\[
\Ed(\GR)_{ij}=\{C\in\Iso_\ell(\A_i)\mid [\A_i/C]=[\A_j]\}.
\]
}
\end{definition1}
\noindent The adjacency matrix $\Ad(\GR)_{ij}=\#\Ed(\GR)_{ij}$ is a constant
row-sum matrix by \eqref{grid}:
\begin{equation}
\label{sunset}
\sum_{j=1}^h\#\Ed(\GR)_{ij}=\prod_{k=1}^g(\ell^k +1).
\end{equation}

This paper studies these three $(\ell)^g$-isogeny graphs via definite
quaternion algebras. 
It naturally divides into two parts -- Part 1
(Sections \ref{def} --\ref{brr}) develops this 
infrastructure on definite quaternion
algebras; Part 2 (Sections \ref{revenue} -- \ref{Ram})  connects the quaternion infrastructure
to superspecial abelian varieties together with their polarizations and 
isogenies, and 
then applies it to our three isogeny graphs.
In Section \ref{def} we prove the foundational material required on 
the arithmetic of definite quaternion algebras together with the Hermitian forms and unitary groups defined from them.
Section \ref{br} introduces the Brandt matrices $B_g(\ell)$ for the
maximal order $\cO$ of $\HH_p$, first
defined for $g>1$ in the 1980's by Hashimoto, Ibukiyama,
Ihara, and Shimizu -- see \cite{Ha}.
Gross's algebraic modular forms \cite{Gr1} for the quaternionic unitary
group subsequently provided a more general context for these matrices.
 In Section \ref{brr} we extend Brandt matrices to Brandt graphs
$\BR_{\!g}(\ell,p)$ and $\br_{\!g}(\ell,p)$;
we further extend Brandt graphs to Brandt simplicial complexes
in \cite{jz}. 
Brandt graphs, like Brandt matrices,
 are defined entirely
in terms of definite quaternion algebras
and as such are amenable to machine computation.
Brandt graphs contain slightly more 
information than Brandt matrices --
the  Brandt matrix  $B_g(\ell)$ is the adjacency matrix of
the big Brandt graph $\BR_{\!g}(\ell,p)$ and the weighted adjacency
matrix of the little Brandt graph with weights $\br_{\!g}(\ell,p)$
(Proposition \ref{sail}).

In Part \ref{smile2} we turn to algebraic geometry.  We
consider
superspecial abelian varieties, their polarizations, and their isogenies
in Section \ref{revenue}.
We introduce the key  notion of an $[\ell]$-polarized abelian  variety
and its $[\ell]$-dual.  
Section \ref{ble} then
defines
the three $(\ell)^g$-isogeny graphs
$\GR_{\!g}(\ell,p)$, $\gr_{\!g}(\ell,p)$, and $\wgr_{\!g}(\ell,p)$.
Sections \ref{co} -- 
\ref{Ram} contain our  main results on isogeny graphs, 
which we now summarize.

{\bf A. Relationship between the quaternion infrastructure and our isogeny graphs.}
We prove in Theorem \ref{adbig}
 the fundamental result that big isogeny graph is the big
Brandt graph: $\GR_{\!g}(\ell,p)=\BR_{\!g}(\ell, p)$.
Likewise the little isogeny graph with weights is the little Brandt
graph with weights: $\gr_{\!g}(\ell,p)=\br_{\!g}(\ell,p)$ (Theorem \ref{smad}).
We further explain how to get the enhanced isogeny graph $\wgr_{\!g}(\ell,p)$
from the little isogeny graph $\gr_{\!g}(\ell,p)$ in Theorem \ref{dealing}.
Because of these theorems our three isogeny graphs can all be defined
and computed entirely in terms of definite quaternion algebras -- it is never
necessary to write down superspecial abelian varieties or isogenies.
In Section \ref{Ram}, we
compute our isogeny graphs  for a range of $174$ 
triples $(g,\ell,p)$
with $g=2,\,3$
including
$13$ examples with $g=3$ -- an impossible feat working with
explicit superspecial abelian varieties and $(\ell)^g$-isogenies.

{\bf B. Connectedness theorems.}\hspace*{1em}It is well known that
the $\ell$-isogeny graph $\Gr_{1}(\ell,p)$ for supersingular elliptic curves in characteristic
$p$ is connected.  A main theorem of this paper is
that the isogeny graphs 
$\GR_{\!g}(\ell,p)$, $\gr_{\!g}(\ell,p)$, and $\wgr_{\!g}(\ell,p)$ are 
connected for $g\geq 1$; cf.  Section \ref{co}.
This had been conjectured for $g=\ell=2$ in \cite[Conjecture 1]{CDS}, 
for example;
we establish the result here for all $\ell$ and $g\geq 1$.
  Additionally we prove
that $\GR_{\!g}(\ell,p)$ and $\gr_{\!g}(\ell,p)$ are not bipartite.
Besides results on polarizations, the main ingredients of the proof for $g>1$ are strong approximation
for the quaternionic unitary group (Theorem \ref{g>1}) and 
Theorem \ref{elect}
on factoring isogenies which in turn follows
from Theorem \ref{elect1} on the symplectic group $\Sp_{2g}$ over $\Z/\ell^n\Z$.
Note that knowing strong approximation still requires the results
on factoring isogenies to deduce connectedness.
The quaternionic unitary group has previously been applied to moduli
of abelian varieties in characteristic $p$ by Chai \cite[Prop.~1]{Ch},
Ekedahl/Oort \cite[\S 7]{EO},
and Chai/Oort \cite[Prop.~4.3]{CO}; a version of strong approximation for the 
quaternionic unitary group is given in \cite[Lemma 7.9]{EO}.

{\bf C. \boldmath{$\ell$}-adic uniformization;
Shimura curves when \boldmath{$g=1$}.}\hspace*{1em}Let
$\Gamma_0=\cO[1/\ell]^\times$ viewed as a subgroup of $\GL_2(\Q_\ell)$
with $\overline{\Gamma}_0$ its image in
$\PGL_2(\Q_\ell)$.  Similarly let $\Gamma_1=\{\gamma\in\Gamma_0\mid
\Norm_{\HH_p/\Q}(\gamma)=1\}$ with $\overline{\Gamma}_1$ its image in 
$\PGL_2(\Q_\ell)$.  Let $\Delta=\Delta_\ell$ be the Bruhat-Tits tree
for $\SL_2(\Q_\ell)=\Sp_2(\Q_\ell)$.  We prove that 
$\gr_1(\ell,p)=\Gamma_0\backslash
\Delta_\ell$ and $\wgr_1(\ell,p)=\Gamma_1\backslash\Delta_\ell$ as graphs with
weights in Theorem \ref{twooo}. We then generalize this to $g>1$ in
Theorem \ref{Dos}: Let $\cSS_{2g}$ be the special $1$-skeleton of the 
Bruhat-Tits building $\cB_{2g}$ for the symplectic group
$\Sp_{2g}(\Q_\ell)$ as in Remark \ref{motel}.  Let 
 $\rU_g(\cO[1/\ell])$ be the quaternionic unitary group with 
$\GU_g(\cO[1/\ell])$ the general quaternionic unitary group as in
\eqref{wet}.  Then we prove
$\gr_{\!g}(\ell,p)=\GU_g(\cO[1/\ell])\backslash \cSS_{2g}$
and $\wgr_{\!g}(\ell,p)=\rU_g(\cO[1/\ell])\backslash \cSS_{2g}$ as graphs with
weights---see Theorem \ref{Dos}.

When $g=1$ we can use this result to connect the $\ell$-isogeny graph
$\wgr_{1}(\ell,p)$ for supersingular elliptic curves in characteristic $p$
to the bad reduction of Shimura curves.  Let $B$ be the rational quaternion
algebra of discriminant $\ell p$ with $\cM\subset B$ a maximal order.
Let $V_B/\Q$ be the Shimura curve parametrizing abelian surfaces with
quaternionic multiplication (QM) by $\cM$ with $M_B/\Z$ the coarse moduli
scheme model for $V_B /\Q$ constructed by Drinfeld \cite{drin}. Then
$M_B\times\Z_\ell$ is an {\em{admissible curve}} in the sense of
\cite[Defn.~3.1]{jl1}, and so has a dual graph 
\cite[Defn.~3.2]{jl1} $\GG(M_B\times\Z_\ell/\Z_\ell)$
which is a graph with lengths as in Definition \ref{weight1}\eqref{weight12}.
We show in Corollary \ref{font1}\eqref{font11} that
$\GG(M_B\times\Z_\ell/\Z_\ell)=\wgr_1(\ell,p)$.  But the dual graph
$\GG(M_B\times\Z_\ell/Z_\ell)$ governs vanishing cycles on the curve
$M_B\times \Z_\ell/\Z_\ell$: the character group of the N\'{e}ron model
of the jacobian $\Jac(V_B)/\Q_\ell$ is 
$H_{1}(\GG(M_B\times\Z_\ell/\Z_\ell),\Z)$.
The fact that the dual graph of the Shimura curve $V_B$ in characteristic
$\ell$ is an isogeny graph for supersingular elliptic curves
in the different characteristic $p$ is the key to Ribet's proof
\cite{Rib} of Serre's Conjecture ``Epsilon'', and so ultimately to 
Fermat's Last Theorem.

Generalizing this picture to $g>1$ is compelling: Relate
$\gr_{\!g}(\ell,p)$, $\wgr_{\!g}(\ell,p)$ to vanishing cycles for 
higher-dimensional Shimura varieties over $\Q_\ell$.

{\bf D. The Ramanujan property for $\GR_{\!g}(\ell,p)$.}\hspace*{1em}
The big isogeny graph $\GR_{\!g}(\ell,p)$ is a regular graph, and 
one can ask whether it is Ramanujan.  If $g=1$ it is always 
Ramanujan, as follows from the Riemann hypothesis for curves over finite
fields.  
Hence naively one might expect the Ramanujan property to continue
to hold for $g\geq 2$ -- see, for example, \cite[Hypothesis 1]{CS}.
The adjacency matrix $\Ad(\GR_{\!g}(\ell,p))$ is the Brandt
matrix $B_g(\ell)$, and so amenable to machine computation
as discussed in {\bf A} above.
In Section \ref{Ram} we give the results of checking the Ramanujan
property over a range of $\ell$ and $p$ with $g=2,\,3$.  The 
memory requirements grow rapidly with $\ell$ and especially $g$;
we had no computations finish with $g>3$. We computed $174$
examples with $g>1$ and $2$  or more vertices
 and found only $4$ Ramanujan: $(g,\ell, p)=
(2,2,5),\, (2,2,7), \,(2,3,7), \,(3,2,3)$ are Ramanujan.
They all have two vertices, although not every $2$-vertex
$\GR_{\!g}(\ell,p)$ is Ramanujan. 
So seemingly for $g\geq 2$ the isogeny
graph $\GR_{\!g}(\ell, p)$ is generically {\em not} Ramanujan.
In Section \ref{Ram1}
 we compute $\GR_2(2,11)$ in terms of superspecial abelian
surfaces and Richelot isogenies, thereby giving
a non-Ramanujan example computed entirely by algebraic geometry. 
In Section \ref{Ram2} we likewise compute $\GR_2(2,7)$ in terms
of abelian surfaces and Richelot isogenies to give a Ramanujan
example computed entirely via algebraic geometry.

We conclude the introduction with brief comments on prior results.
The case $g=1$ was the setting for multiple proposals in
post-quantum cryptography, and naturally the question of generalizing
to $g>1$ arose.  Castryck, Decru, and Smith \cite{CDS}  proposed the superspecial isogeny graph $\GR_2(2,p)$ as a good generalization to abelian surfaces.
Previous work often concentrates on $\GR_2(2, p)$ where computations are
feasible using classical
Richelot isogenies -- see, for example,
Katsura and Takashima \cite{KK} and the references therein.
(In contrast, we compute $\GR_g(\ell, p)$ by computing Brandt matrices
for quaternion algebras.) The paper \cite{ATY} gives an alternate
definition of $\GR_g(\ell,p)$ and develops this.

\part{The quaternion infrastructure}

\label{smile}
\section{Definite rational quaternion algebras}
\label{def}

Let $\H$ be a definite quaternion algebra over $\Q$ with a maximal
order $\cO_\H$, main involution
$x\mapsto \overline{x}$, and reduced norm
$\Norm_{\H/\Q}(x)=\Norm(x)=x\overline{x}$.
Set $\H_1^\times=\{h\in\H^\times \mid \Norm_{\H/\Q}(x)=1\}$.
The reduced norm $\Norm:\H\rightarrow \Q$ generalizes to the reduced norm 
$\Norm: \Mat_{g\times g}(\H)\rightarrow \Q$
(given by a  multiplicative polynomial of degree $2g$ in the entries of
the matrix).
Put
\begin{equation}
\label{dotted}
\SL_g(\cO_\H)=\{M\in\MM(\cO_\H)\mid \Norm(M)=1\}
\end{equation}
with $\SL_g(\H)$ defined analogously.
Note that 
\[
\SL_g(\cO_\H)=\GL_g(\cO_\H)=\{M\in\MM(\cO_\H)\mid M \text{ is invertible}\}.
\]

Let $\widehat{\Z}=\varprojlim \Z/n\Z$ be the profinite completion
of $\Z$ and 
$\widehat{\Q}=\widehat{\Z}\otimes\Q$ the finite ad\`eles of
$\Q$.  Then $\cO_{\widehat{\H}}=\cO_{\H}\otimes \widehat{\Z}$ is the profinite
completion of $\cO_\H$ and $\widehat{\H}=\cO_{\widehat{\H}}\otimes \Q$ is the 
finite ad\`{e}les of $\H$.

\subsection{Hermitian matrices}
\label{her}

Let $g\ge 1$ be an integer.  
 A matrix $H\in \MM(\cO_\H)$ is {\sf Hermitian}
if $H^\dagger \colonequals \overline{H}^t =H$.  
Set
\begin{equation}
\label{ox}
\msH_g(\cO_\H)=\{H\in\Mat_{g\times g}(\cO_\H)\mid H \text{ is positive-definite Hermitian}\}.
\end{equation}
The ``Haupt norm'' $\HNm$ of Braun-Koecher \cite[Chap.~2, \S 4]{BK}
(see also \cite[Thm.~6 and proof, \S 21]{Mu}) is defined
on Hermitian matrices in $\Mat_{g\times g}(\HH)$ and  gives a map
$\HNm:\msH_g(\cO_\H)\rightarrow \NN$. It is characterized by
$\HNm(\Id_{g\times g})=1$ and $\Norm(H)=\HNm(H)^2$ for a Hermitian matrix 
$H\in\Mat_{g \times g}(\HH)$; see \cite[p.~152, 153]{Ek}, where $\HNm$
is denoted $\Pf$ and is defined via the usual Pfaffian on skew-symmetric 
matrices.
For an integer $d\geq 1$ put
\begin{equation}
\label{hermit}
\msH_{g,d}(\cO_\H)=\{H\in\msH_g(\cO_\H)\mid \HNm(H)=d\}.
\end{equation}
The group $\SL_g(\cO_\H)=\GL_g(\cO_\H)$ acts on $\msH_{g,d}(\cO_\H)$ by
$H\cdot M=M^\dagger H M$. Set
\begin{equation}
\label{class}
\overline{\msH}_{\!\!g,d}(\cO_\H)\colonequals \msH_{g,d}(\cO_\H)/\SL_g(\cO_\H)
\end{equation}
with $[H]\in\overline{\msH}_{\!\!g,d}(\cO_\H)$ the class defined by 
$H\in\msH_{g,d}(\cO_\H)$. The sets $\overline{\msH}_{\!\!g,d}(\cO_\H)$
for $d\geq 1$ are finite.

\subsection{Strong Approximation for \texorpdfstring{\except{toc}{\boldmath{$\H_1^\times$}}\for{toc}{$\H_1^\times$}}{\unichar{"210D}\unichar{"2081}\unichar{"00D7}}}
\label{rotor}

We now give the statement of strong  approximation followed by
several consequences for the multiplicative group of norm-$1$ quaternions.
In Section \ref{co} we will use strong approximation for
the quaternionic unitary group.
\subsubsection{Strong approximation}
\label{str}
Let $k$ be an algebraic number field with $\infty$ the set of all
archimedean places of $k$. Let
$S\supseteq\infty$ be a finite set of places of $k$. 
Let $G$ be a linear algebraic group over $k$.
Let $G_{\bA}$ be the ad\`{e}le group of $G$,
$G_{S}\subset G_{\bA}$ be the $S$-component $\prod_{v\in S}G_{k_v}$ of $G_{\bA}$,
and $G_{k}\subset G_{\bA}$ be the $k$-rational points of $G$ embedded
diagonally.
\begin{definition1} 
{\rm
The pair $(G,S)$ has {\sf strong approximation}
if $G_{S}G_{k}$ is dense in $G_{\bA}$.
}
\end{definition1}

Say that a connected noncommutative linear algebraic group $G$ 
over a field $k$ is 
$k$-{\sf simple} if it has no positive-dimensional proper
normal subgroups.
We now give a statement of Strong Approximation sufficient for our
purposes, quoting Platonov and Rapinchuk~\cite[Thm.~7.12]{PR}.
The general result is due to Kneser~\cite{K}.

\begin{theorem}
\label{sun}
Let $G$ be a  simply connected and $k$-simple linear algebraic
group over a number field $k$.  Suppose $G_{S}$ is not compact.
Then $(G, S)$ has strong
approximation.
\end{theorem}

\subsubsection{A key lemma}
\label{lem}
\begin{lemma}\label{l0}
Let $\H/\Q$ be an arbitrary definite quaternion algebra with
maximal order $\cO_\H$, let $I$ be a fractional right
$\cO_\H$-ideal of norm $1$, and let $\ell$ be a prime unramified in
$\H$.  Then 
there exists an element in $I\otimes\Z[1/\ell]$ of norm $1$.
\end{lemma}
\begin{proof}
Let $G$ be the algebraic group over $\Q$ associated to 
$\H_1^\times=\{\beta\in\H^\times\mid \Norm(\beta)=1\}$;
then $G(\Q)=\H_1^\times$.
The algebraic group $G$ is
  simply connected with a simple Lie algebra since $G(\R)\cong\SU(2)$
  is.  (Let $G'$ be the algebraic group over $\Q$ assoicated
to $\H^\times$, so that $G'(\Q)=\H^\times$ and $G'(\R)=(\H\otimes_\Q\R)^\times$
is the multiplicative group of Hamilton real quaternions.
Note that $G'$  doesn't satisfy the hypotheses of
  Theorem~\ref{sun} (Strong Approximation): $G'(\R)=(\H\otimes_\Q\R)^\times$ is
  topologically 
$\mathbb{R}^4$ minus the origin, which is simply connected, but as a
  Lie group, $G'(\R)$ is $\mathbb{R}_{>0}^{\times}\times
  \SU(2)$, which is not simple.)  Let $S=\{\ell,\infty\}$.  By
  Theorem \ref{sun}, $(G,S)$ has strong approximation.

Now consider the subset $U\subset G_\bA$ given by the local
  conditions that at each prime $q\ne\ell$ we have $\beta\in
  (I\otimes\Z_q) \cap G(\Q_q)$.  Notice that this local condition is
the standard one that $\beta\in(\cO_\H \otimes \Z_q)^\times$
at all finite primes away from the numerator and denominator
of the fractional ideal $I$, hence $U$ is open.  That $U$ is
  nonempty follows from the fact that every right ideal in a
  quaternion algebra is locally principal.  Hence we see that
  $U\cap G_{S}G_\Q$ is nonempty and there exists some $\beta\in (I\otimes
  \Z[1/\ell]) \cap \H_1^\times$.
\end{proof}
\begin{lemma}\label{l1}
For each positive integer $x$ and any prime $\ell$ not ramified in
$\H$ there exists an element in $\cO_\H[1/\ell]$ of norm $x$.
\end{lemma}
\begin{proof}
Let $I$ be an (integral) right ideal of $\cO_\H$ of norm $x$ and
$\alpha$ an element of $\H$ also of norm $x$.  Then $\alpha^{-1}I$ has
norm $1$ and we may apply Lemma \ref{l0} to obtain a
$\beta\in\alpha^{-1}(I\otimes\Z[1/\ell])$ of norm $1$.  Then
$\alpha\beta$ has norm $x$ and $\alpha\beta\in
I\otimes\Z[1/\ell]\subset\cO_\H[1/\ell]$.
\end{proof}

\subsubsection{Consequences for \texorpdfstring{$\MM(\cO_\H)$}{Matg\unichar{"00D7}g(O\unichar{"210D})}}

\begin{lemma}\label{l2}
For any prime $q$ and any Hermitian \mbox{$H\in\MM
  (\cO_\H\otimes\Z_{(q)})$} which is positive definite
of reduced norm $1$, there is a matrix
$M\in\MM (\cO_\H\otimes\Z_{(q)})$ such that $H=M^\dagger M$.
The matrix $M$ satisfies $\Norm M=1$.
\end{lemma}
\begin{proof}
Since $H$ has reduced norm $1$, there exists some
$v\in(\cO_\H\otimes\Z_{(q)})^g$ such that $x=v^\dagger Hv\in\Z_{(q)}$ satisfies
$x\notin q\Z_{(q)}\subset \Z_{(q)}$.
By positive-definiteness $x>0$ and after scaling $v$
we may assume $x^{-1}$ is an integer.  Then by applying Lemma \ref{l1}
for $\ell$ away from $q$ and the ramified primes of $\H$ there exists
an $\alpha\in\cO_\H \otimes\Z_{(q)}$ of norm $x^{-1}$.

The proof is by induction on $g$.  The assertion is trivial for $g=1$:
here $H=M=1\in \cO_\H\otimes \Z_{(q)}$. For a general $g$, 
let $v_1=v\alpha$ as above.  Note that $v_1^\dagger Hv_1=1$ 
and consider $\langle v_1\rangle^\bot=\{w\in(\cO_\H\otimes \Z_{(q)})^g\mid
v_1^\dagger H w=0\}$. The Hermitian form defined by $H$ restricts to 
a positive definite Hermitian form of reduced norm $1$ on 
$\langle v_1\rangle^\bot$, so
we're reduced
to showing the theorem on $\langle v_1\rangle^\bot\cong
(\cO_\H\otimes \Z_{(q)})^{g-1}$.

Finally we have
\[
1=\Norm(H)=\Norm(M^\dagger)\Norm(M)=\Norm(M)^2,
\]
so $\Norm(M)=1$ since $\Norm(M)$ is positive.
\end{proof}

\subsection{The quaternionic unitary group}
\label{quat}

If $B$ is an algebra with anti-involution having fixed ring $R$
and $M^\dagger$ is the conjugate-transpose defined using the 
anti-involution for $M\in\MM(B)$, set
\begin{equation}
\label{wet}
\begin{split}
\rU_g(B) & = \{M\in\MM(B)\mid M^\dagger M=\Id_{g\times g}\}\\
\GU_g(B)  & =\{M\in\MM(B)\mid M^\dagger M=\lambda \Id_{g\times g}\text{ with }\lambda\in R^\times\}.
\end{split}
\end{equation}
For a Hermitian matrix $H\in\Mat_{g\times g}(\cO_\H)$ 
set
\[
\rU_{H}(\cO_\H)=\{M\in\MM(\cO_\H)\mid M^\dagger H M=H\}.
\]
Let $H_0\in\msH_{g,1}(\cO_\H)$ be the Hermitian matrix $\Id_{g\times g}$.
Then $\rU_{H_0}(\cO_\H)=\rU_g(\cO_\H)$ as in \eqref{wet}.

Let $L\subset\H^g$ be a finitely
generated right $\cO_\H$-submodule such that $L\otimes\Q\cong \H^g$.
Such an $L$ is {\sf{principally polarized}} 
if there exists a $c\in\Q^\times$ such that $cH_0$
restricted to $L$ is $\cO_\H$-valued and unimodular.  
We define the {\sf{dual}} of $L$ to be 
\begin{equation}
\label{stuffing2}
\widehat{L}=c^{-1}L.
\end{equation}

\begin{remark1}\label{hat1}
{\rm
Notice that this agrees with the standard definition of dual with
respect to a pairing; thus, if $L\subset L'$ then
$\widehat{L'}\subset\widehat{L}$ and
$[L':L]=[\widehat{L}:\widehat{L'}]$.
}
\end{remark1}

\begin{theorem}
\label{pumpkin}
For $M\in\GU_g(\Hh)$, set $\gamma(M)$ equal to the principally
polarized right $\cO_\H$-submodule of $\H^g$ given by
$\gamma(M)=M\cO_\Hh^g\cap\H^g$. The association $M\mapsto \gamma(M)$
induces a one-to-one correspondence between $\GU_g(\Hh)/\GU_g(\cO_\Hh)$
and the set of principally polarized right $\cO_\H$-submodules of $\H^g$.
\end{theorem}
\begin{proof}
  The  module $\gamma(M)$ is principally
polarized since after tensoring with $\widehat{\Z}$, the Hermitian
form is given by $M^\dagger M$ which is 
the identity times a scalar in $\widehat{\Q}^\times$, which can be approximated
by an element of $\Q^\times$.

This map is well defined since if $MU$ with $U\in\GU_g(\cO_\Hh)$ is another
representative of the same class in $\GU_g(\Hh)/\GU_g(\cO_\Hh)$, then
$U\cO_\Hh^g=\cO_\Hh^g$.  Hence
$MU\cO_\Hh^g\cap\H^g=M\cO_\Hh^g\cap\H^g$.
It is injective since if
$M\cO_\Hh^g\cap\H^g=M'\cO_\Hh^g\cap\H^g$, we must have $MN=M'$ for
some $N\in\GL_g(\cO_\Hh)$.  But we also have
$N=M'M^{-1}\in\GU_g(\Hh)$.  Therefore,
$$N\in\GL_g(\cO_\Hh)\cap\GU_g(\Hh)=\GU_g(\cO_\Hh),$$ and $[M]=[M']$.

Finally, to see that this map is surjective, let $L$ be a principally
polarized right $\cO_\H$-submodule.  Since all finitely generated
modules over $\cO_\H$ are locally free, $L$ is given by
$N\cO_\Hh^g\cap\H^g$ for some $N\in\GL_g(\Hh)$.  The Hermitian form on
$L\otimes\widehat{\Z}$ is given by $N^\dagger N$, and since $L$ is
principally polarized, $cN^\dagger N$ is $\cO_\Hh$-valued and
unimodular for some $c\in\Q^\times$.  However, since all integral
unimodular Hermitian forms are locally trivial 
(as follows, for example, from Lemma \ref{l2}),
 there exists a
$V\in\GL_g(\cO_\Hh)$ such that $V^\dagger cN^\dagger NV$ is the
identity.  So we can set $M=NV$ and have $M\in\GU_g(\Hh)$ with
$M\cO_\Hh^g\cap\H^g=N\cO_\Hh^g\cap\H^g=L$.
\end{proof}

\begin{definition1}
\label{later}
{\rm 
We define the {\sf{classes}} of $\GU_g(\cO_\H)$, denoted
$\cP_{\!g}(\cO_\H)$, to be the equivalence classes of principally polarized
right $\cO_\H$-submodules of $\H^g$ up to left multiplication by $\GU_g(\H)$.
Hence there is a one-to-one correspondence between 
$\GU_g(\H)\backslash \GU_g(\Hh)/\GU_g(\cO_\Hh)$ and 
$\cP_{\!g}(\cO_\H)$ induced by the map $\gamma$ of Theorem \ref{pumpkin}:
\begin{equation}
\label{feeler}
\cP_{\!g}(\cO_\H)\cong \GU_g(\H)\backslash \GU_g(\Hh)/\GU_g(\cO_\Hh).
\end{equation}
There are a finite  number of classes of $\GU_g(\cO_\H)$.
We will call $\#\cP_{\!g}(\cO_\H)$ the {\sf{class number}}.  It is
independent of the choice of maximal order $\cO_\H$ since the
  isomorphism class of $\cO_\Hh$ is independent of $\cO_\H$ and
consequently we denote it by $h_g(\H)$. Let the principally polarized
right $\cO_\H$-module $L_i$ be a representative of the class $[L_i]\in
\cP_g(\cO_\H)$ for $1\leq i\leq h=h_g(\H)$.  }
\end{definition1}

\begin{remark1}\label{hat2}
{\rm
  Notice that $L$ and $\widehat{L}$ belong to the same class.  Also
  note that for $U\in\GU_g(\H)$ we have
  $\widehat{UL}=(U^{-1})^\dagger\widehat{L}$.
}
\end{remark1}

\noindent 
When $g=1$ we recover the standard description of the ideal classes
$\cP_1(\cO_\H)$ 
and the class number $h(\H)$:
\begin{equation}
\label{bonus}
\cP_1(\cO_\H)\cong \H^\times\backslash \Hh^\times/\cO_\Hh^\times\quad\text{and}\quad
h_1(\H)=h(\H)=\# \,\H^\times\backslash \Hh^\times/\cO_\Hh^\times ,
\end{equation}
cf. \cite[\S 3.5.B]{V}.

\begin{theorem}\label{t2}
If $g>1$, then  $\omsH_{\!\!g,1}(\cO_\H)$ is in
one-to-one correspondence with the classes of $\GU_g(\cO_\H)$, or
equivalently with the double
cosets $$\GU_g(\H)\bs\GU_g(\Hh)/\GU_g(\cO_\Hh).$$
\end{theorem}
\begin{proof}
We first define the map
\begin{equation}
\label{squash}
\iota: \omsH_{\!\!g,1}(\cO_{\H}) \to \GU_g(\H)\bs\GU_g(\Hh)/\GU_g(\cO_\Hh)
\end{equation}
by the following procedure: for $[H]\in\omsH_{\!\!g,1}(\cO_\H)$ write $H=M^\dagger
M$ for $M\in\SL_g(\H)$, such an $M$ exists by Lemma \ref{l2}. For
each prime $q$ write $H=N_q^\dagger N_q$ with
$N_q\in\SL_g(\cO_\H\otimes\Z_q)$ (again these exist by Lemma
\ref{l2}). Let $N=(N_q)\in\SL_g(\cO_\Hh)$, and notice that
$(MN^{-1})^\dagger MN^{-1}=I$, so $MN^{-1}\in\rU_g(\Hh)$. 
Set $\iota(H)=MN^{-1}$ and define the map $\iota$ in \eqref{squash}
by
\begin{equation}
\label{iota}
  \iota([H])=[\iota(H)]=[MN^{-1}].
\end{equation}

We must now prove that the map  $\iota$ in \eqref{iota}
is well defined, injective, and
surjective.\\
\emph{Well-definedness:} Suppose $M'\in\SL_g(\H)$ is another choice of
$M$ and $N'\in\SL_g(\cO_\Hh)$ another choice of $N$ satisfying
$H={M'}^\dagger M'={N'}^\dagger N'$.  Then $M'M^{-1}\in\rU_g(\H)$ and
$N{N'}^{-1}\in\rU_g(\cO_\Hh)$.  Hence
$$[M'{N'}^{-1}]=[(M'M^{-1})(M{N}^{-1})(N{N'}^{-1})]$$ corresponds to the
same class as $[M{N}^{-1}]$ in
$\GU_g(\H)\bs\GU_g(\Hh)/\GU_g(\cO_\Hh)$.

Now suppose $H'\in\msH_{g,1}(\cO_\H)$ is another representative of the same
class as $H$ in $\omsH_{\!\!g,1}(\cO_\H)$, i.e., $H'=U^\dagger HU$ for some
$U\in\SL_g(\cO_\H)$.  Thus if $H=M^\dagger M=N^\dagger N$ with
$M\in\SL_g(\H)$ and $N\in\SL_g(\cO_\Hh)$, then $$H'=(MU)^\dagger
MU=(NU)^\dagger NU$$
with $MU\in\SL_g(\H)$, $NU\in\SL_g(\cO_\Hh)$,
and $MU(NU)^{-1}=MN^{-1}$.

\emph{Injectivity:} Suppose $\iota([H])=\iota([H'])$.  Let
$H=M^\dagger M=N^\dagger N$ and $H'=M'^\dagger M'=N'^\dagger N'$ with
$M,M'\in\SL_g(\H)$ and $N,N'\in\SL_g(\cO_\Hh)$.  Thus
$MN^{-1}=VM'{N'}^{-1}W^{-1}$ with $V\in\GU_g(\H)$ and
$W\in\GU_g(\cO_\Hh)$.  Set $V^\dagger V=vI$ for $v\in\Q^\times$ and
$W^\dagger W=wI$ for $w\in\widehat{\Z}^\times$.  Let
$U=M^{-1}VM'=N^{-1}WN'$ and observe that
$U\in\GL_g(\H)\cap\GL_g(\cO_\Hh)=\SL_g(\cO_\H)$.  Now
$$H\cdot
U=U^\dagger H U=(M^{-1}VM')^\dagger M^\dagger MM^{-1}VM'={M'}^\dagger
V^\dagger V M'=vH' ,$$
and a similar argument shows $H\cdot U=wH'$.
Hence $v=w\in\Q^\times\cap\widehat{\Z}^\times=\Z^\times$. So $v=\pm1$
and we can rule out $-1$ since $\H$ is definite.  Thus $[H]=[H']$.

\emph{Surjectivity:} Let
$$[V]\in\GU_g(\H)\bs\GU_g(\Hh)/\GU_g(\cO_\Hh)$$
with $V\in\GU_g(\Hh)$.  Put $V^\dagger V=vI$ for $v\in\widehat{\Q}$.
Put $v=ab$ with $a\in\Q_{>0}$ and $b\in\widehat{\Z}^\times$.  Then
by Lemma \ref{l1} there exist $\alpha\in\H$ with $\N(\alpha)=a$ and
$\beta\in\cO_\Hh^\times$ with $N(\beta)=b$.  After replacing $V$ with
$\alpha^{-1}V\beta^{-1}$ we may assume
$V\in\rU_g(\Hh)\subset\SL_g(\Hh)$.

We will apply strong approximation to $G=\SL_g(\H)$ with
$S=\{\infty\}$.  Note that $G_{\infty}$ is not compact for $g>1$ and
hence the pair $(G, S)$ satisfies the conditions of Theorem \ref{sun}.
The local
conditions we will impose at each prime $q$ will be
$V^{-1}M\in\SL(\cO_\H\otimes\Z_q)$.  These are the standard conditions
away from finitely many primes and are trivially nonempty since $V$
always satisfies them.

Hence there exists $M\in\SL_g(\H)$ such that
$N=V^{-1}M\in\SL_g(\cO_\Hh)$.  Thus $\iota([M^\dagger M])=[V]$.
\end{proof}

\section{Brandt matrices}
\label{br}

\subsection{Definition of Brandt matrices}
\label{br1}
Set $h_g\colonequals h_g(\H)$ with
$\cP_g(\cO_\H)=\{[L_1],\ldots,[L_{h_g}]\}$
for principally polarized right $\cO_\H$-modules
$L_1,\ldots , L_{h_g}\subseteq \H^g$. We can now define the
{\sf Brandt matrix} $B_g(n)\in\Mat_{h_g\times h_g}(\Z)$ for a natural number $n$.
\begin{definition1}
\label{Brandt}
{\rm
Let $g\geq 1$ and $n\in\NN$.
For $1\leq j\leq h_g(\H)\equalscolon h_g$, set 
\begin{align}
\nonumber E_j(g) &\colonequals \{U\in\GU_g(\H)\mid L_j=UL_j\},\\
\nonumber e_j(g) &\colonequals \# E_j(g),\\
\label{dubious}
\widetilde{\mathbf{B}}_g(n)_{ij}&\colonequals 
\{U\in\GU_g(\H)\mid [L_i:UL_j]=n^{2g}\},\\
\nonumber E(U)&\colonequals \{V\in\GU_g(\H)\mid VL_i=L_i\text{ and }
VUL_j=UL_j\}\text{ for } U\in\widetilde{\mathbf{B}}_g(n)_{ij},\\
\nonumber e(U)&\colonequals \# E(U)\text{ for } 
U\in\widetilde{\mathbf{B}}_g(n)_{ij}.
\end{align}
The sets $E_j(g),\, \widetilde{\mathbf{B}}_g(n)_{ij}, \,E(U)$
above depend on the choice of representatives $L_1,\ldots , L_{h_g}$.
However, they change by an explicit one-to-one
correspondence if we change the representatives: if $\tilde{L}_j=W_jL_j$
for $W_j\in\GU_g(\H)$, then in \eqref{dubious} the $W_j$ can be
absorbed into the $U$. In particular $\#\widetilde{\mathbf{B}}_g(n)_{ij }$
and $\#E_j(g)$ do not depend on the choice of representatives for
$\cP_g(\cO_\H)$.

Note that $E_j(g) \leq \GU_g(\H)$ and $E(U)\leq E_i(g)$ 
for $U\in\widetilde{\mathbf{B}}_g(n)_{ij}$.
Define the equivalence relation $\sim_b$ on 
$\widetilde{\mathbf{B}}_g(n)_{ij}$
by $U\sim_b U'$ if $U=U'V$ for some $V\in E_j(g)$. 
Likewise define the equivalence relation $\sim_l$ on
$\widetilde{\mathbf{B}}_g(n)_{ij}$ by
$U\sim_l U'$ if $U=V_iU'V_j$ for some $V_j\in E_j(g)$
and some $V_i\in E_i(g)$.
Define the big quotient
\begin{equation}
\label{stuffing}
\mathbf{B}_g(n)_{ij}^{\rm big}=\widetilde{\mathbf{B}}_g(n)_{ij}/\!\sim_b
\end{equation}
and the little quotient
\begin{equation}
\label{stuffing1}
\mathbf{B}_g(n)_{ij}^{\rm little}=\widetilde{\mathbf{B}}_g(n)_{ij}/\!\sim_l.
\end{equation}
Let
$[U]_b\in\mathbf{B}_g(n)_{ij}^{\rm big}$,
$[U]_l\in\mathbf{B}_g(n)_{ij}^{\rm little}$ 
be the equivalence classes of
$U\in\widetilde{\mathbf{B}}_g(n)_{ij}$.
An equivalent definition is
\begin{equation}
\label{stuffing4}
\mathbf{B}_g(n)_{ij}^{\rm big}=\{L'_j\mid [L_i:L_j']=n^{2g} \text{ and } UL_j=L_j'
\text{ for some }U\in\GU_g(\HH)\}.
\end{equation}
We define an equivalence relation $\sim$ on 
$\mathbf{B}_g(n)_{ij}^{\rm big}$ in \eqref{stuffing4} 
by $L'_{j_1}\sim L'_{j_2}$ 
if there exists $U\in\GU_g(\H)$ such that $UL_i=L_i$
and  $UL'_{j_1}=L'_{j_2}$ ;
write $[L'_j]_l$ for the equivalence class of 
$L_j'\in\mathbf{B}_g(n)_{ij}^{\rm big}$.
Then we equivalently have
\begin{equation}
\label{stuffing5}
\mathbf{B}_g(n)_{ij}^{\rm little}=\left(\mathbf{B}_g(n)_{ij}^{\rm big}/\!\sim\right) =
\left\{[L'_j]_l\mid L'_j\in \mathbf{B}_g(n)_{ij}^{\rm big}\right\} .
\end{equation}

For $1\leq i,j\leq h_g$, put
\begin{equation}
\label{rose}
  B_g(n)_{ij}= \#\mathbf{B}_g(n)_{ij}^{\rm big}=
\frac{\#\widetilde{\mathbf{B}}_g(n)_{ij }}
  {\#E_j(g)}\quad\text{and}\quad B_g(0)_{ij}=1/e_j(g).
\end{equation}
The matrices $B_g(n)$ do not depend
on the choice of maximal order $\cO_\H$ or on the choice of
representatives for $\cP_g(\cO_\H)$, up to the obvious
indeterminacy of simultaneously permuting the rows and columns.
Let $\BB_g(\cO_{\HH})\subseteq \Mat_{h_g\times h_g}(\Z)$
be the $\Z$-algebra generated by $B_g(n)$, $n\geq 1$.
The $\Z$-algebra $\BB_g(\cO_{\HH})$ does not depend on the choice
of maximal order $\cO_{\HH}$ and hence we can denote it as
$\BB_g=\BB_g(\HH)$.
}
\end{definition1}

\begin{remark1}
\label{Brandt well}
{\rm
In the classical case $g=1$,
$h=h_1(\HH)$ is the class number of $\H$.
Let $I_1$, \dots, $I_h$ be representatives for the right $\cO_\H$-ideal
classes and let $\OO_i$ be the left order of $I_i$, $1\leq i\leq h$.
We have  $e_i=e_i(1)=\#\OO_i^\times$.
Definition \ref{Brandt} in the
special case $g=1$ gives 
\begin{align*}
\label{duck}
E_j(1)&=\cO_j^\times,\\
e_j\colonequals e_j(1)&=\#\cO_j^\times,\\
\widetilde{\mathbf{B}}_1(n)_{ij}&=
\{\lambda\in I_iI_j^{-1}\mid\Norm\lambda=n\Norm(I_iI_j^{-1})\},\\
B_1(n)_{ij}&=\frac{\#\{\lambda\in I_iI_j^{-1}\mid\Norm\lambda=
n\Norm(I_iI_j^{-1})\}}{e_j},\\
E(\lambda)&=\{u\in\cO_i^\times\mid \lambda^{-1}u\lambda\in\cO_j^\times\}\quad\text{ for }
\lambda\in \widetilde{\mathbf{B}}_1(n)_{ij},\\
e(\lambda)&=\# E(\lambda)\quad\text{ for }\lambda\in 
\widetilde{\mathbf{B}}_1(n)_{ij}.
\end{align*}
In particular, $B_1(1)=\Id_{h\times h}$ and $B_1(n)\in\Mat_{h\times h}(\Z)$ for 
$n\geq 1$.
The Brandt matrix $B_1(0)$ is $B_1(0)_{ij}=1/e_j$ and 
$\BB_1=\BB_1(\HH)\subseteq\Mat_{h\times h}(\Z)$ is the $\Z$-algebra
generated by the Brandt matrices $B_1(n)$, $n\geq 1$.
}
\end{remark1}

By Theorem \ref{pumpkin}, 
\[
\#\omsH_{\!\!g,1}(\cO_\H)=h_g(\H)\equalscolon h_g=\#\cP_g(\cO_\H).
\]
Write
\begin{align}
\omsH_{\!\!g,1}(\cO_\H)&=\{[H_1],\ldots ,[H_{h_g}]\}\quad\textup{for
$H_i\in\msH_{\!\!g,1}$, $1\leq i\leq h_g$,  and}\label{bid}\\
\cP_g(\cO_\H)&=\{[L_1],\ldots, [L_{h_g}]\}\quad\textup{with
  $L_i$ a principally polarized right $\cO_\H$-submodule of $\H^g$}.
\nonumber
\end{align}
First we give an equivalent definition of the Brandt matrix in terms of
$\omsH_{\!\!g,1}(\cO_\H)$ in case $g>1$. It is convenient to make the
following definition.
\begin{definition1}
  \label{groan}
  Suppose $H, H'\in\omsH_{\!\!g,1}(\cO_\H)$.  For a natural number $n$
  set
\begin{align*}
  \mathbf{U}_n(H,H')& \colonequals \{M\in \Mat_{g\times g}(\cO_\H)\mid
  M^\dagger H M=nH'\}\text{ and}\\
  \mathbf{U}(H)&\colonequals \mathbf{U}_1(H,H).
\end{align*}
Note that $\mathbf{U}(H)$ acts on $\mathbf{U}_n(H,H')$ by
multiplication on the left and $\mathbf{U}(H')$ acts on $\mathbf{U}_n(H,H')$
by multiplication on the right. Define an equivalence relation
$\sim_b$ on $\mathbf{U}_n(H, H')$ by $M\sim_bMU'$ for $U'\in
\mathbf{U}(H')$ and set
$\mathbf{U}_n(H, H')^{\text{\rm big}}\colonequals \mathbf{U}_n(H,H')/
\mathord{\sim}_b$
with $[M]\in\mathbf{U}_n(H, H')^{\text{\rm big}}$
the class defined by $M\in\mathbf{U}_n(H, H')$.
Define an equivalence relation $\sim_l$ on $\mathbf{U}_n(H,H')^{\text{\rm big}}$
by $[M]\sim_l [UM]$ for $U\in\mathbf{U}(H)$.
Set $\mathbf{U}_n(H,H')^{\text{\rm little}}\colonequals \mathbf{U}_n(H,H')^
{\text{\rm big}}/\mathord{\sim}_l$.
\end{definition1}
\begin{theorem}\label{Brandt2}

  Let $g>1$ with $\gamma$ as in Theorem \textup{\ref{pumpkin}} and 
$\iota$ as in \eqref{iota}.
If $[\gamma(\iota([H_i]))]=[L_i]$ and
$[\gamma(\iota(H_j))]=[L_j]$, then we have equivalently
\begin{align*}
  \mathbf{B}_g(n)_{ij}^{\text{\rm big}} & =\mathbf{U}_n(H_i,H_j)^{\text{\rm big}}
    \text{  and}\\
    \mathbf{B}_g(n)_{ij}^{\text{\rm little}} & = \mathbf{U}_n(H_i,H_j)^
           {\text{\rm little}}.
\end{align*}
In particular we have

\begin{align*}
e_j(g)& =\#\mathbf{U}(H_j) \text{ and}\\
B_g(n)_{ij} & = \mathbf{B}_g(n)_{ij}^{\text{\rm big}}
=\frac{\#\mathbf{U}_n(H_i, H_j)}
{e_j(g)}
\end{align*}
  for $n\geq 1$.
\end{theorem}
It clearly suffices prove the following lemma.

\begin{lemma}
\label{yeast}
  Let $g>1$.  Choose any $[H_1],\,[H_2]\in\omsH_{\!\!g,1}(\cO_\H)$.  Let
  $[\gamma(\iota(H_k))]=[L_k]$ for $k\in\{1,2\}$.  There exists a
  bijective correspondence between $\{U\in\GU_g(\H)\mid
  [L_1:UL_2]=n^{2g}\}$ and $\{B\in\MM(\cO_\H)\mid B^\dagger
  H_1B=nH_2\}$.
\end{lemma}
\begin{proof}
  For $k\in\{1,2\}$, let $[V_k]=\iota([H_k])$ with $L_k=\gamma(V_k)$.
  Let $H_k=M_k^\dagger M_k=N_k^\dagger N_k$, with $M_k\in\SL_g(\H)$,
  $N_k\in\SL_g(\cO_\Hh)$, and $V_k=M_kN_k^{-1}\in\rU_g(\Hh)$ as in
  (\ref{iota}).

  For $B\in\MM(\cO_\H)$ with $B^\dagger H_1B=nH_2$, let
  $U_B=M_1BM_2^{-1}$.  Notice that $B^\dagger M_1^\dagger
  M_1B=nM_2^\dagger M_2$ so $U_B^\dagger U_B=n\Id_{g\times g}$ and
  $U_B\in\GU_g(\H)$.  Similarly take $W_B=N_1BN_2^{-1}$ and observe
  that $W_B\in\GU_g(\Hh)$ and $W_B\in\MM(\cO_\Hh)$. Taking reduced
  norms, we get $\Norm(W_B)=n^g$. 
  Therefore, $W_B\cO_\Hh^g\subset\cO_\Hh^g$ with
  \begin{equation}\label{ind 1}
  [\cO_\Hh^g:W_B\cO_\Hh^g]=n^{2g}.
  \end{equation}
  Notice that $U_BV_2=M_1BN_2^{-1}=V_1W_B$.  Apply this to \eqref{ind
    1} gives
  $n^{2g}=[V_1\cO_\Hh^g:V_1W_B\cO_\Hh^g]=[V_1\cO_\Hh^g:U_BV_2\cO_\Hh^g]$.
  And intersecting with $\H^g$ gives $n^{2g}=[L_1:U_BL_2]$.

  The correspondence $B\mapsto U_B$ is clearly well-defined and
  injective.  We will now show it is surjective.  Given
  $U\in\GU_g(\H)$ with $[L_1:UL_2]=n^{2g}$, hence tensoring with
  $\cO_\Hh$ we see $[V_1\cO_\Hh^g:UV_2\cO_\Hh^g]=n^{2g}$.  Let
  $B=M_1^{-1}UM_2$ and $W=N_1BN_2^{-1}$, so $UV_2=M_1BN_2^{-1}=V_1W$.
  Hence, $n^{2g}=[V_1\cO_\Hh^g:V_1W\cO_\Hh^g]=[\cO_\Hh^g:W\cO_\Hh^g]$.
  Thus, $W\in\MM(\cO_\Hh)$ with $\Norm(W)=n^g$.  Since the
  $M_k$'s and $N_k$'s have reduced norm $1$, $U$ also has reduced
  norm $n^g$; hence, $U^\dagger U=n\Id_{g\times g}$ since
  $U\in\GU_g(\H)$.  Also that means that $B\in\MM(\H)$.
But also $B=N_1^{-1}WN_2$ with $N_1, N_2\in \SL_2(\cO_{\Hh})$ and
$W\in\MM(\cO_{\Hh})$, so $B\in\MM(\cO_{\Hh})$ and hence
$B\in\MM(\cO_\H)$.
 Straightforward algebra shows that $B^\dagger H_1B=nH_2$, and we are
  done.
\end{proof}

\subsection{Brandt matrices: Examples}
\label{exa}
The Brandt matrices $B_g(n)$ for a maximal order $\cO_\HH\subseteq\HH$
are amenable to machine computation, although the memory requirements
rapidly grow with $n$ and especially $g$ so that 
few examples are accessible
with $g=3$.  We had no computations finish for $g\geq 4$.

\subsubsection{\texorpdfstring{$\HH=\HH_7$}{\unichar{"210D}=\unichar{"210D}\unichar{"2085}}}
Take $\HH=\HH_7$, the rational definite quaternion algebra of discriminant $7$.
The first class numbers of $\HH_7$ are: $h_1(\HH_7)=1$,
$h_2(\HH_7)=2$, $h_3(\HH_7)=5$.
The Brandt matrices $B_g(\ell)$ are given in Table 1 below for
primes $\ell=2, 3, 5, 11$ and $1\leq g\leq 3$.
Note that in all cases 
$B_g(\ell)$
has constant row-sum $N_g(\ell)=\prod_{k=1}^g(1+\ell^k)$
in keeping with Theorem \ref{hen}\eqref{hen1}.
A ? in the table below means that the computation did not finish.\\[.1in]
\begin{center}
\begin{table}[h]
\begin{tabular}{l c c c c}
   & \boldmath{$B_g(2)$} & \boldmath{$B_g(3)$} & 
\boldmath{$B_g(5)$} & \boldmath{$B_g(11)$}\\[.15in]
 \boldmath{$g=1$} & $[3]$ & $[4]$ & $[6]$ & $[12]$ \\[.2in]
\boldmath{$g=2$} & $\begin{bmatrix} 11 & 4 \\ 6 & 9\end{bmatrix}$ &
 $\begin{bmatrix} 28 & 12 \\ 18 & 22\end{bmatrix}$ &
$\begin{bmatrix} 112 & 44 \\ 66 & 90\end{bmatrix}$ &
$\begin{bmatrix} 928 & 536 \\ 804 & 660\end{bmatrix}$\\[.3in]
\boldmath{$g=3$} &  $\begin{bmatrix} 45 & 36 & 8 & 32 &14 \\ 
18 & 27 & 6 & 60 & 24\\14 & 21 & 30 & 14 & 56\\4 & 15 & 1 & 101 & 14\\
7 & 24 & 16 & 56 & 32\end{bmatrix}$ &
$\begin{bmatrix} 208 & 208 & 0 & 640 & 64 \\ 
104 & 184 & 32 & 640 & 160\\0 & 112 & 112 & 616 & 280\\ 80 & 160 & 44 & 
676 &1 60\\
32 & 160 & 80 & 640 & 208\end{bmatrix}$ & \text{\Large ?} &   
\text{\Large ?}\\[.6in]
\end{tabular}
\caption{\label{bt}Brandt matrices $B_g(\ell)$ for $\HH_7$}
\end{table}
\end{center}

\subsubsection{\texorpdfstring{$\HH=\HH_{11}$}{\unichar{"210D}=\unichar{"210D}\unichar{"2081}\unichar{"2081}}}
Now take $\HH=\HH_{11}$, the rational definite 
quaternion algebra of discriminant $11$.
The first class numbers of $\HH_{11}$ are: $h_1(\HH_{11})=2$,
$h_2(\HH_{11})=5$, $h_3(\HH_{11})=19$.
Table 2 below gives the Brandt matrices $B_g(\ell)$ for
$\ell=2,3,5,7$ and $g=1, 2$.  
Again in all examples $B_g(\ell)$
has constant row-sum $N_g(\ell)=\prod_{k=1}^g(1+\ell^k)$.\\[.1in]
\begin{center}
\begin{table}[h]
\begin{tabular}{l c c c c}
   & \boldmath{$B_g(2)$} & \boldmath{$B_g(3)$} & 
\boldmath{$B_g(5)$} & \boldmath{$B_g(7)$}\\[.15in]
 \boldmath{$g=1$} & $\begin{bmatrix} 1 &2\\3&0\end{bmatrix}$ & $\begin{bmatrix} 2 & 2\\3 &1\end{bmatrix}$ & 
 $\begin{bmatrix}4 &2\\3 & 3\end{bmatrix} $ & $\begin{bmatrix} 4 & 4\\6 &2\end{bmatrix}$ \\[.3in]
\boldmath{$g=2$} & $\begin{bmatrix} 
    3 & 4 & 4 & 0 & 4\\
    3 & 6 & 0 & 6 & 0\\
    3 & 0 & 3 & 8 & 1\\
    0 & 3 & 4 & 8 & 0\\
    9 & 0 & 3 & 0 & 3\end{bmatrix}$ &
$\begin{bmatrix}
    8  & 8  & 4 & 16 &  4\\
    6 & 20  & 0 & 12  & 2\\
    3 &  0 &  9&  22 &  6\\
    6  & 6 &  11 &  16 &  1\\
    9 &  6 &  18 &  6 &  1\end{bmatrix}$ &
$\begin{bmatrix}
    36 &  32 &  36&  32&  20\\
    24 & 42 & 24 & 60 &  6\\
    27 &  24 & 41 & 58 &  6\\
    12 & 30 &  29 &  78 &   7\\
    45 &  18 &  18&  42 & 33\end{bmatrix}$ &
$\begin{bmatrix}
    80 &  80 &  72 &  128 &  40\\
    60 & 128 &  48 & 144 &  20\\
    54 &  48 &  94 &  172 &  32\\
    48 &  72&   86&  176 &  18\\
    90 &  60 &  96 & 108 &  46\end{bmatrix}$\\[.6in]
\end{tabular}\\[.15in]
\caption{\label{bt1}Brandt matrices $B_g(\ell)$ for $\HH_{11}$}
\end{table}
\end{center}

\subsection{First properties of Brandt matrices}
\label{fi}
To simplify the discussion, we restrict to the case of the definite
quaternion algebra $\H=\H_p$ ramified at one finite prime $p$
and choose a maximal order $\cO=\cO_{\H_p}\subseteq\H_p$.

\subsubsection{The classical case: \texorpdfstring{\except{toc}{$g=1$}\for{toc}{$g=1$}}{g=1}}
\label{wolf}

We start by reviewing known properties of the classical Brandt
matrices $B(n)=B_1(n)$ for $\OO=\OO_{\HH_p}$, $n\geq 0$, 
largely following  Gross \cite[\S 1, 2]{Gr}.
Set $h=h_1(\H_p)$. Almost all the results given are due to 
Eichler \cite{E}.

\begin{remark1}
\label{cat}
{\rm
\begin{enumerate}[\upshape (a)]
\item
\label{cat1}
For $n\geq 0$ with $(n,p)=1$ the row sums $\sum_j B(n)_{ij}$ are
independent of $i$.  For $n\geq 1$ and $\ell\neq p$ prime
with $N_1(\ell)$ as in \eqref{grid1},
\[
\sum_j B(\ell)_{ij}=N_1(\ell)\colonequals \ell +1.
\]
\item
\label{cat2}
If $(m,n)=1$, then $B(mn)=B(m)B(n)$.
\item
\label{cat3}
$B(p)$ is a permutation matrix with $B(p)^2=\Id_{h\times h}$
and $B(p)^k=B(p^k)$.
\item
\label{cat4}
For a prime $\ell\neq p$ and $k\geq 2$,
\[
B(\ell^k)=B(\ell^{k-1})B(\ell)-\ell  B(\ell^{k-2}).
\]
\item
\label{cat5}
Set $e_j=e_j(1)$ for $1\leq j\leq h$.
We have
$e_jB(n)_{ij}=e_iB(n)_{ji}$ for $1\leq i, j\leq h$.
Equivalently, let $v_1, \ldots, v_h$ be the 
standard basis of $\Z^h$.
Define the inner product
$\langle v_i, v_j\rangle=e_i\delta_{ij}$ on $\Z^h$.
Then the Brandt matrices $B(n)$, $n\geq 1$, are self-adjoint
with respect to $\langle \,\, , \,\,\rangle$.
\item
\label{cat6}
(Eichler's mass formula)
Let $\H=\H_p$.  Then
\[
\sum_{i=1}^{h} \frac{1}{e_i}=\frac{p-1}{24}.
\]
Equivalently, the sum of any row of $B(0)$ is $(p-1)/24$ with
$B(0)=B_1(0)$ as in \eqref{rose}.
\item
\label{cat7} For all $m$ and $n$ we have $B(m)B(n)=B(n)B(m)$.
\item
\label{cat8}
The commutative $\Q$-algebra $\BB\otimes \Q$ is semi-simple, and
isomorphic to the product of totally real number fields.
\end{enumerate}
}
\end{remark1}

\subsubsection{The general case \texorpdfstring{\except{toc}{$g\geq 1$}\for{toc}{$g\geq 1$}}{g\unichar{"2265}1}}
\label{dog2}
We now generalize the results of Remark \ref{cat} to the
Brandt matrices $B_g(n)$ of Definition \ref{Brandt} with $\H=\H_p$.
Put $h_g=h_g(\H)$ and let $\BB_g=\BB_g(\H)$ be the $\Z$-subalgebra
of $\Mat_{h_g\times h_g}(\Z)$ generated by the Brandt matrices $B_g(n)$
for $n\geq 1$ as in Definition \ref{Brandt}.
\begin{theorem}
\label{hen}

\begin{enumerate}[\upshape (a)]
\item
\label{hen1}
For $n\geq 1$ with $(n,p)=1$, $\sum_j B_g(n)_{ij}=N_g(n)$
as in \eqref{grid}. For $n=\ell\neq p$ prime this gives
\[
\sum_j B_g(\ell)_{ij}=N_g(\ell)\colonequals \prod_{k=1}^g(1+\ell^k)
\]
with $N_g(\ell)$ as in \eqref{grid1}.
In particular, the row sums $\sum_jB_g(n)_{ij}$ are
independent of $i$ and in fact only depend on $n$ and $g$
\textup{(}they do not depend on $p$\textup{)}.
\item
\label{hen2}
If $(m,n)=1$, then $B_g(mn)=B_g(m)B_g(n)$.
\item
\label{hen5}
We have
$e_j(g)B_g(n)_{ij}=e_i(g)B_g(n)_{ji}$ for $1\leq i, j\leq h_g$.
Equivalently, let $v_1, \ldots, v_{h_g}$ be the 
standard basis of $\Z^{h_g}$.
Define the inner product
$\langle v_i, v_j\rangle_g=e_i(g)\delta_{ij}$ on $\Z^{h_g}$.
Then the generalized Brandt matrices $B_g(n)$, $n\geq 1$, are self-adjoint
with respect to $\langle \,\,\, , \,\,\rangle_g$.
\item
\label{hen6}
\textup{(Mass formula of Ekedahl 
and Hashimoto/Ibukiyama)}
\[
M_g:=\sum_{i=1}^{h_g} \frac{1}{e_i(g)}=\frac{(-1)^{g(g+1)/2}}{2^g}\left\{
\prod_{k=1}^g \zeta(1-2k)\right\}\cdot\prod_{k=1}^g\{p^k+(-1)^k\}.
\]
Equivalently, the sum of any row of $B_g(0)$
as in \eqref{rose}
is $M_g$.
Note that for $g=1$ we have $M_1=(p-1)/24$
and so recover Theorem \textup{\ref{cat}\eqref{cat6}}.

\item
\label{hen7}
For all $m$ and $n$ we have $B_g(m)B_g(n)=B_g(n)B_g(m)$.

\item
\label{hen8}
The commutative $\Q$-algebra $\BB_g\otimes \Q$ is semi-simple, and
isomorphic to the product of totally real number fields.

\end{enumerate}
\end{theorem}
\begin{proof}
\eqref{hen2}: It's not hard to see
that $$(B_g(m)B_g(n))_{ij}=\frac{\#\{U\in\GU_g(\H) \text{ and $L_k$
    p.~p.}\mid [L_i:L_k]=m^{2g} \text{ and } [L_k:UL_j]=n^{2g}\}}
{e_j(g)} ,$$
where p.~p. denotes principally polarized.
Since $m$ and $n$ are relatively prime given $L_i$, $L_j$,
and $U$ with $[L_i:UL_j]=(mn)^{2g}$ there exists a unique principally
polarized $L_k$ with $[L_i:L_k]=m^{2g}$ and $[L_k:UL_j]=n^{2g}$.  Thus
$(B_g(m)B_g(n))_{ij}=B_g(mn)_{ij}$.\\

\eqref{hen1}: By \eqref{hen2} we may restrict to the case when
$n=\ell^r$ is a prime power with $\ell\neq p$.  Since each p.~p.~ lattice belongs to
precisely one equivalence class, we have using \eqref{stuffing4} and
\eqref{rose}
\begin{align}
  \nonumber\sum_j B_g(\ell^r)_{ij}&=\sum_j\# \{L_j^\prime \text{ p.~p.} \mid [L_i:L_j^\prime]=(\ell^r)^{2g}=\ell^{2rg}
  \text{ and $UL_j=L_j^\prime$ for some $U\in\GU_g(\HH) \}$}\\
  \label{risky}
    &=\#\{L_j \text{ p.~p.}\mid [L_i:L_j]=\ell^{2rg}\}.
\end{align}

We now suppose $g>1$ so that classes of $\GU_g(\OO_\H)$ can be
described by Hermitian matrices as well as principally polarized lattices
as in Theorem \ref{t2}.  The case $g=1$ is covered by
the classical results of Eichler in Remark \ref{cat}.

Set $\OO_\ell=\OO_{\HH}\otimes \Z_\ell$ for $\ell\neq p$
prime and $\HH_\ell=\HH\otimes \Q_\ell$.
There is a well-known one-to-one correspondence $\leftrightarrow$
between $g\times g$ quaternionic Hermitian matrices and
$2g\times 2g$ symplectic ($=$ nondegenerate alternating)
matrices; a reference is
\cite[\S 1]{Ek}. 
Let $e=\left[\begin{smallmatrix}0 & 1\\-1 &0
  \end{smallmatrix}\right]$ and
identify $\OO_\ell$ with $\Mat_{2\times 2}(\Z_\ell)$
so that the main involution on $\OO_\ell$
becomes
\begin{equation}
  \label{desk}
\overline{\left[\begin{smallmatrix} a & b \\c & d
    \end{smallmatrix}\right]}=\left[\begin{smallmatrix}
    d & -b\\ -c &a\end{smallmatrix}\right]=e^{-1}
\left[\begin{smallmatrix} a & b \\c & d
  \end{smallmatrix}\right]^t e.
\end{equation}
Let $E$ be the $2g\times 2g$ block matrix with $e$'s on the diagonal.
The identification of $\OO_\ell$ with $\Mat_{2\times 2}(\Z_\ell)$ gives
an identification of $\Mat_{g\times g}(\OO_\ell)$ with
$\Mat_{g\times g}(\Mat_{2\times2}(\Z_\ell))=\Mat_{2g\times 2g}(\Z_\ell)$;
as notation  $A\in\Mat_{g\times g}(\OO_\ell)$ is identified with
$\tilde{A}\in\Mat_{2g\times 2g}(\Z_\ell)$ so $\Norm (A)=\det (\tilde{A})$.
For a $A\in \Mat_{g\times g}(\OO_\ell)$ with $A^\dagger=\overline{A}^t
\in\Mat_{g\times g}(\OO_\ell)$,
\eqref{desk} implies that
\begin{equation}
  \label{desk1}
  \widetilde{A^\dagger} =E^{-1}\tilde{A}^tE .
  \end{equation}
In case $A=H$ is Hermitian so that $H^\dagger =H$, \eqref{desk1}
gives
\begin{equation*}
  E\tilde{H} = E\widetilde{H^\dagger} = \tilde{H}^t E,
\end{equation*}
so that
\begin{equation}
  \label{desk2}
  (E\tilde{H})^t=\tilde{H}^tE^t=\tilde{H}^t(-E)=-\tilde{H}^tE=-(E\tilde{H}).
\end{equation}
This gives the one-to-one correspondence $\leftrightarrow$: to
the Hermitian matrix $H\in\Mat_{g\times g}(\OO_\ell)$ we associate
the symplectic matrix $S_H\colonequals E\tilde{H}\in\Mat_{2g\times 2g}(\Z_\ell)$.
With $\Pf$ denoting the Pfaffian of a symplectic
matrix we have $\HNm(H)=\Pf(S_H)$.

We examine how this correspondence behaves with respect to sublattices.
Let $L$ be a nondegenerate Hermitian right $\OO_\ell$-module
of rank $g$ such that $L\otimes \Q_\ell\cong \HH_\ell^g$ with Hermitian
form given by
$H\in\Mat_{g\times g}(\OO_\ell)$
(so $H^\dagger =H$).
Let $L'=AL\subset L$ be an $\OO_\ell$-sublattice of finite index $i=[L:L']$
for $A\in\Mat_{g\times g}(\OO_\ell)$; then $(i)=(\Norm(A)^2)$ as ideals
in $\Z_\ell$ and $i=\ell^{2\val_\ell(\Norm(A))}$.
Let $H'=A^\dagger H A$ be the restriction of
$H$ to $L'$. We have
\begin{equation}
  \label{greasy}
  \Norm(H')=\Norm(A)^2\Norm(H)=i\Norm(H), \text{ or, }\left| \HNm(H')\right|=\sqrt{i} \left|
  \HNm(H)\right|.
\end{equation}

Separately, let $\tilde{L}$ be a rank-$2g$ symplectic $\Z_\ell$-lattice with symplectic
form $S$. Let
$\tilde{L}'\colonequals M\tilde{L}\subseteq \tilde{L}$
for $M\in\Mat_{2g\times 2g}(\Z_\ell)$ 
be a sublattice of finite index $\tilde{\iota}$ so $(\tilde{\iota})=
(\det(M))$ as ideals in $\Z_\ell$
with symplectic form $S'=M^t S M $ given by restricting $S$ to $\tilde{L}'$.
Then $(\Pf(S'))=(\Pf(S)\tilde{\iota})$ as ideals in $\Z_\ell$.
We have
\begin{align}
  \label{greasy1}
 \nonumber (\tilde{\iota})&=(\det(M))\subseteq \Z_\ell,
  \,\,(\det(S'))=(\det(M)^2\det(S)),\text{ and }\\
  (\Pf(S')&=(\det(M)\Pf(S))=(\tilde{\iota}\Pf(S))\subseteq \Z_\ell.
\end{align}

Now consider $\tilde{L}=\Z_\ell^{2g}$ with $S=S_H$ and
$A\in\Mat_{g\times g}(\OO_\ell)$
with $H'=A^\dagger HA$ the restriction of
$H$ to the $\OO_\ell$-sublattice
$L'=AL$,  and take $M=\tilde{A}\in\Mat_{2g\times 2g}(\Z_\ell)$.
Then $S'=\tilde{A}^tS\tilde{A}$ the restriction of $S'$ to
the $\Z_\ell$-sublattice $\tilde{L}'=\tilde{A}\tilde{L}$.  Note that
\begin{equation}
  \label{beast}
  S_{H'}=
E\tilde{A^\dagger}\tilde{H}\tilde{A}=
  \tilde{A}^t E\tilde{H}\tilde{A}=\tilde{A}^tS_H \tilde{A}=S'
\end{equation}
using \eqref{desk1}. Using \eqref{greasy}, the indices $\iota = [L:L']$ and
$\tilde{\iota}=[\tilde{L}:\tilde{L}']$ are related by
\begin{equation}
  \label{sunny}
  \iota=\ell^{2\val_\ell(\Norm(A))}=\ell^{2\val_\ell(\det(\tilde{A}))}=
  \bigl(\ell^{\val_\ell(\det{\tilde{A}})}\bigr)^2=\tilde{\iota}^2.
  \end{equation}

The problem of computing a given row sum of the Brandt matrix
$B(\ell^r)$ by \eqref{risky} is the following: We are given
$L=\OO_{\ell}^g$ together with a unimodular
Hermitian form $H\in\Mat_{g\times g}(\OO_\ell)
$ (so $\left| \HNm(H)\right|=1$).  We have to
count the $\OO_\ell$-submodules $\ell^r L\subset L'\subset L$
with $[L:L']=\ell^{2rg}$ such that the restriction $H'$ of $H$
to $L'$ satisfies $H'=\ell^{r}\mathbf{H'}$ with
$\mathbf{H'}\in\Mat_{g\times g}(\OO_\ell)$ unimodular.
Applying the one-to-one correspondence $\leftrightarrow$ above
induced by $H\leftrightarrow S=S_H$ and $H'\leftrightarrow S'=S_{H'}$
we see that this equivalent to the following computation with
symplectic $\Z_\ell$-lattices of rank $2g$: given the lattice
$\tilde{L}=\Z_\ell^{2g}$ together with a unimodular symplectic pairing
$S$ (unimodular in the sense that $\Pf(S)\in \Z_\ell^\times$),
count the $\Z_\ell$-sublattices $\ell^r\tilde{L}\subset \tilde{L}'\subset
\tilde{L}$ with $[\tilde{L}:\tilde{L}']=\ell^{rg}$ such that
the restriction $S'$ of $S$ to $\tilde{L'}$ satisfies
$S'=\ell^r\mathbf{S}'$ for a unimodular symplectic matrix
$\mathbf{S}'\in \Mat_{2g\times 2g}(\Z_\ell)$. (The indices
$[L:L']$ and $[\tilde{L}:\tilde{L}']$ are related by \eqref{sunny}.)
It follows that
$\tilde{L}'/\ell^r\tilde{L}$ is a maximal isotropic subspace
of $\tilde{L}/\ell^r\tilde{L}$ with respect to the
$\Z/\ell^r\Z$-symplectic pairing on $\tilde{L}/\ell^r\tilde{L}$
induced by $S$.  Moreover,
given the maximal isotropic subspace $\tilde{L}'/\ell^r\tilde{L}$
we can recover $\tilde{L}'\subseteq \tilde{L}$.  We now
remark that all unimodular symplectic lattices over $\Z_\ell$
of the same dimension are isomorphic.  Hence without loss
of generality we can start with $\tilde{L}=\Z_\ell^{2g}$
and $S$ the standard unimodular symplectic pairing.  So 
the sum of the entries in any row of the Brandt
matrix $B_g(\ell^r)$ is equal to $N_{g}(\ell^r)$
as in \eqref{grid}.  In particular all the row sums of $B_g(\ell^r)$
are equal and (perhaps surprisingly) do not depend on $p$.

\eqref{hen5}: By definition this is equivalent to proving that
$\#\{U\in\GU_g(\H)\mid [L_i:UL_j]=n^{2g}\}$ and $\#\{U\in\GU_g(\H)\mid
      [L_j:UL_i]=n^{2g}\}$ are equal.  By the comments following \eqref{rose}
we can replace the $L$'s with arbitrary representatives of their classes.  By
Remark \ref{hat2} we can use their duals, so
$$\#\{U\in\GU_g(\H)\mid [L_j:UL_i]=n^{2g}\}=\#\{U\in\GU_g(\H)\mid
[\widehat{L_j}:\widehat{(U^{-1})^\dagger L_i}]=n^{2g}\}.$$ But by
Remark \ref{hat1} the right hand side of the above is equal to
$$\#\{U\in\GU_g(\H)\mid [(U^{-1})^\dagger
  L_i:L_j]=n^{2g}\}=\#\{U\in\GU_g(\H)\mid [L_i:U^\dagger
  L_j]=n^{2g}\},$$ and we are done. \\


\eqref{hen6}: See \cite[p.~159]{Ek} and \cite[Prop.~9]{HI}, cf.
\cite[Thm.~3.1]{Y}. \\


\eqref{hen7}:
The Brandt matrices $B(n)$ are in image of the Hecke algebra for $(G,K)$ with 
$G=\GU_g(\Hh)$, $K=\GU_g(\cO_\Hh)$ acting on
the lattice $\Z[\cP_g(\cO_\H)]$ with basis $\cP_g(\cO_\H)$,
which is a space of algebraic modular forms in the sense of
\cite{Gr1}.  In fact the 
Brandt matrices $B(n)$ are linear combinations of 
standard Hecke operators in the Hecke algebra.

By \eqref{hen2} we may restrict to the case where both $m$ and $n$ are powers of a prime $\ell$. Commutativity here is implied by commutativity of 
the local Hecke algebra for 
\begin{equation}
\label{turn}
G=G_{\ell}=\GU_g(\H\otimes\Z_{\ell}),\quad
K=K_\ell=\GU_g(\cO_\H\otimes\Z_\ell)
\end{equation}
  Satake proves a structure
theorem for this local algebra \cite[Thm. 8]{Sat} which in particular
shows that it is commutative.

Below we give a simple argument for commutativity, worked out 
in correspondence with Guy Henniart and Marie-France Vign\'{e}ras.
We use $(G,K)$ as in \eqref{turn}. 
By  Gelfand's trick \cite[{\rm IV}, \S 1, Thm.~1]{La}
(see also \cite[Prop.~3.8]{Sh1}),
 it suffices to show that for all elements $M\in G$ we have
$
KMK=KM^\dagger K$.

In the case when $\ell=p$, by \cite[Prop.~3.10]{Sh} we have that for every $M\in G$,
$
KMK=KdK
$
for some diagonal matrix $d$ over the quaternion algebra $\H\otimes\Z_\ell$.  We
also have
$
K d K=Kd^\dagger K
$
since in the ramified case all ideals are two-sided and principal
powers of the unique prime ideal.

When $\ell\ne p$, the group $G$ is just the
symplectic group and we can use \cite[Prop.~1.6]{Sh} to show that
for arbitrary $M\in G$ we have
$
KMK=KdK,
$
where $d$ is now in a diagonal matrix over $\Q_\ell$, hence trivially preserved
by transpose.

The only complication is making sure (conjugate-)transpose is in fact
an anti-involution in the basis from \cite{Sh}.  However if
$H$ is the matrix giving the Hermitian (or symplectic) form in Shimura's
basis then we have
$
H^\dagger H H=H$,
so $H$ is itself an element of $G$.  And since for any matrix $M\in G$ we have
$M^\dagger=H M^{-1} H^{-1}$,
(conjugate-)transpose is in fact an anti-involution.\\

\eqref{hen8}: This follows trivially from \eqref{hen5}, \eqref{hen7},
and the fact that self-adjoint matrices are semi-simple with 
real eigenvalues.

\end{proof}
\noindent We do not know the analogue of Remark \ref{cat}\eqref{cat4} for
  our generalized Brandt matrices $B_g(n)$.  For a weak result see
  \cite[Thm.~3]{andrianov}.

\section{The big and little Brandt graphs}
\label{brr}

It is convenient to reformulate Section \ref{br} on Brandt matrices
in the broader
context of Brandt graphs.  We begin with a general discussion
of graphs in order to be precise about the definitions.  We will
again use this in Section \ref{ble} when we consider the big, little,
and enhanced isogeny graphs. 
\subsection{Graphs}
\label{grs}
\begin{definition1}
\label{grdef}
{\rm
A graph $\Grr$ has a set of vertices $\Ver(\Grr)=\{v_1,\ldots , v_s\}$
and a set of (directed) edges $\Ed(\Grr)$.  An edge
$e\in\Ed(\Grr)$ has initial vertex $o(e)$ and terminal
vertex $t(e)$.  For vertices $v_i,\,v_j\in \Ver(\Grr)$, put
\[
\Ed(\Grr)_{ij}=\{e\in\Ed(\Grr)\mid o(e)=v_i \text{ and }
t(e)=v_j\}.
\]
The {\sf adjacency matrix} $\Ad(\Grr)\in\Mat_{s\times s}(\Z)$is the 
matrix with 
\[
\Ad(\Grr)_{ij}=\#\Ed({\Grr})_{ij}.
\]
We place no further restrictions on our definition of a graph.
Serre \cite{Ser} requires graphs to be {\sf graphs with opposites}:
every directed edge $e\in\Ed(\Grr)$ has an {\sf opposite}
edge $\overline{e}\in\Ed(\Grr)$.
An edge $e$ with $\overline{e}=e$
is called a {\sf half-edge}. Serre forbids half-edges; we will call a graph
satisfying his requirements a {\sf graph without half-edges}.
Kurihara \cite{Kur} relaxes Serre's definition to allow
half-edges giving the notion of a {\sf graph with half-edges}.
(A graph with half-edges may have $\emptyset$ as its set of half-edges,
so every graph without half-edges is a graph with half-edges.)
}
\end{definition1}
\begin{definition1}
\label{weight1}
{\rm
\begin{enumerate}[\upshape (a)]
\item
\label{weight11}
A {\sf graph with weights}, or a {\sf weighted graph},
 is a graph with opposites together with a
{\sf weight function} $\w$ mapping vertices and edges to positive integers
such that for each edge $e$ we have 
$\w(e)=\w(\overline{e})$ and 
$\w(e)|\w(o(e))$ (which implies $\w(e)|\w(t(e))$). 
\item
\label{weight12}
Following \cite[Defn.~3-1]{Kur}, a {\sf graph with lengths} is a graph
with opposites together with a length function $f$ mapping edges
to positive integers satisfying $f(e)=f(\overline{e})$.
A graph with weights defines a graph with lengths by setting the length
of an edge equal to its weight and forgetting the weights of the vertices.
\item
\label{wad}
The {\sf  weighted adjacency matrix} $A_{\w}\colonequals \Adw(\Grr)$
of a weighted graph $\Grr$
with $\Ver(\Grr)=\{v_1,\ldots , v_s\}$ is
\[
(A_{\w})_{ij}=\sum_{e\in\Ed({\Grr})_{ij}}\frac{\w(v_i)}{\w(e)}, \quad 1\leq i,j\leq s.
\]
\item
\label{wad1}
Following \cite[\S 3]{Kur}, if $\Grr$ is a graph with half-edges, denote
by $\Grr^\ast$ the graph obtained by removing the half-edges from $\Grr$.
If $\Grr$ is a graph with weights, then the graph $\Grr^\ast$ with
half-edges removed is also a graph with weights---the weights
are inherited from $\Grr$.  Likewise, if $\Grr$ is a graph with lengths,
then $\Grr^\ast$ is a graph  with inherited lengths.

\end{enumerate}
}
\end{definition1}
\noindent Many authors (especially in computer science) call a graph
with weights what we have called a graph with lengths, and accordingly
have a different notion of a weighted adjacency matrix.

For the remainder of this section, we let $\HH$ be a rational
definite quaternion algebra with maximal order $\cO_{\HH}\subseteq\HH$,
main involution $x\mapsto\overline{x}$, and reduced norm
$\Norm_{\HH/\Q}(x)=x\overline{x}$. Let $\hat{\Z}$ be the profinite
completion of $\Z$ and $\hat{\Q}=\hat{\Z}\otimes\Q$ the finite
ad\`{e}les of $\Q$. Set $\cO_{\hat{\HH}}=\cO_{\HH}\otimes\hat{\Z}$ and
let $\hat{\HH}=\cO_{\hat{\H}}\otimes\Q$ be the finite ad\`{e}les of $\HH$.

The classes of $\GU_g(\cO_\H)$ are
\[
\cP_{\!g}(\cO_\H)=\{[L_1],\ldots , [L_h]\}
\]
with $L_i$ a principally polarized right $\cO_\HH$-module and 
$h=h_g(\H)$ as in Definition \ref{later}.  In case $g=1$ the
principally polarized right $\cO_\HH$-module $L_i$ is just
a right $\cO_\HH$-ideal $I_i$ with left order the maximal ideal
$\cO_i\subseteq \H$.
We will freely use the notation in
Definition \ref{Brandt} and Remark \ref{Brandt well} of Section \ref{br}, which
the reader is advised to review.

\subsection{The big Brandt graph \texorpdfstring{\except{toc}{\boldmath{$\BR_{\!g}(n,\cO_{\HH})$}}\for{toc}{$\BR_{\!g}(n,\cO_{\HH})$}}{Br\unichar{"005F}g(n,O\unichar{"210D})}}
\label{rent}

The vertices of the big Brandt graph $\BR_{\!g}(n)\colonequals
\BR_{\!g}(n,\cO_{\HH})$
are
\[
\Ver(\BR_{\!g}(n))=\cP_g(\cO_{\H})=\{[L_1],\ldots , [L_h]\}.
\]
The directed edges connecting the vertex $[L_i]$ to the 
vertex $[L_j]$ are
\[
\Ed(\BR_{\!g}(n))_{ij}=\mathbf{B}_g(n)_{ij}^{\rm big}
\]
as in  \eqref{stuffing} and \eqref{stuffing4}.
The graph $\BR_{\!g}(n)$ is a graph without opposites.
Moreover it is immediate from \eqref{rose} that the adjacency matrix of 
$\BR_{\!g}(n)$ is
the Brandt matrix $B_g(n)$ for $\cO_{\HH}\subseteq\HH$:
\begin{equation}
\label{grinder}
\Ad(\BR_{\!g}(n))=B_g(n).
\end{equation}
When $g=1$ the big Brandt graph $\BR_1(n)$ is the graph
constructed by Pizer \cite{Piz1}, \cite{Piz2} from 
the classical Brandt matrices.

\subsection{The little Brandt graph \texorpdfstring{\except{toc}{\boldmath{$\br_{\!g}(n,\cO_{\HH})$}}\for{toc}{$\br_{\!g}(n,\cO_{\HH})$}}{br\unichar{"005F}g(n,O\unichar{"210D})}}
The vertices of the little Brandt graph $\br_{\!g}(n)\colonequals
\br_{\!g}(n,\cO_{\HH})$ are
\[
\Ver(\br_{\!g}(n))=\cP_g(\cO_{\H})=\{[L_1], \ldots, [L_h]\}.
\]
The directed edges connecting the vertex $[L_i]$ to the vertex
$[L_j]$ are
\[
\Ed(\br_{\!g}(n))_{ij}=\mathbf{B}_g(n)_{ij}^{\rm little}
\]
as in  \eqref{stuffing1} and \eqref{stuffing5}.

Unlike the big Brandt graph, the little Brandt graph $\br_{\!g}(n)$ {\em
  is} a graph with opposites: the opposite $\overline{e}$ of an edge
$e\in\Ed(\br_{\!g}(n))_{ij}$ with $e=[L'_j]_l$ as in \eqref{stuffing4} is
$\overline{e}=[V\widehat{L_i}]_l$, where $V$ satisfies $V\widehat{L'_j}=L_j$ with
the dual $\widehat{L'}$ of the principally polarized finitely
generated right $\cO_{\HH}$-module $L'$ with $L'\otimes \Q=\HH^g$
defined in \eqref{stuffing2}.  The little Brandt graph $\br_{\!g}(n)$ is
also a graph with weights: we set $\w([L_i])=e_i(g)$ for $[L_i]\in
\cP_g(\cO_{\HH})=\Ver(\br_{\!g}(n))$ and $\w([U]_l)=e(U)$ for $[U]_l\in
\mathbf{B}_g(n)^{\rm little}_{ij}=\Ed(\br_{\!g}(n))_{ij}$ in the notation
\eqref{dubious}. To see that this is well-defined, verify
\begin{enumerate}[\upshape (a)]
\item
$e(U)=e(U')$ if $[U]_l=[U']_l
\in\mathbf{B}_g(n)^{\rm little}_{ij}=\Ed(\br_{\!g}(n))_{ij}$ and 
\item
$\w(e)=\w(\overline{e})$ for $e\in \Ed(\br_{\!g}(n))_{ij}$.
\end{enumerate}

It follows from the definitions that the weighted adjacency matrix 
of the little 
Brandt graph $\br_{\!g}(n)$ is the usual Brandt matrix $B_{\!g}(n)$:
\begin{proposition}
\label{sail}
We have
$\Adw(\br_{\!g}(n))=\Ad(\BR_{\!g}(n))=B_g(n)$.
\end{proposition}

\part{Applying the quaternion infrastructure to isogeny graphs}
\label{smile2}
\section{Superspecial abelian varieties, their principal and 
\texorpdfstring{$[\ell]$}{[\unichar{"2113}]}-polarizations,
and their isogenies}
\label{revenue}

In this section $X$ is an abelian variety defined over a field $k$
(not necessarily algebraically closed) with dual abelian variety
$\hat{X}=\Pic^0(X)$; $A$ will continue to denote a superspecial
  abelian variety.  If $f\colon X\rightarrow Y$ is a morphism of abelian
varieties over $k$, the dual morphism $\hat{f}\colon \hat{Y}\rightarrow
\hat{X}$ is defined over $k$.  For a point $x$ of $X$, denote by $t_x$
translation by $x$ on $X$; the isomorphism class of a line bundle $L$
on $X$ is denoted $[L]$.  A homomorphism $\tau\colon X\rightarrow \hat{X}$
is {\sf symmetric} if $\hat{\tau}= \tau$, where we identify
$X=\hat{\hat{X}}$ via the canonical isomorphism
\begin{equation}
\label{lob}
\kappa_X\colon X\stackrel{\simeq}{\longrightarrow} \hat{\hat{X}}\text{ of \cite[Thm.~7.9]{vM}, for example.}
\end {equation}
A line bundle $L$ on $X$ gives rise to a symmetric homomorphism
$\varphi_L\colon X\rightarrow \hat{X}$ which maps points $x$ of 
$X$ to $[t_x^\ast L\otimes L^{-1}]$. The Poincar\'{e} line bundle
on $X\times\hat{X}$ is denoted $\caP$.
Our standard reference for abelian varieties is
\cite{vM}, whose modern treatment of polarizations is ideally suited to
our needs here.  

\begin{definition1}\text{\rm (cf.~\cite[Cor.~11.5, Defn.~11.6]{vM}.) }
\label{polar}
{\rm
A {\sf polarization} of an abelian variety $X$ over a field $k$
is a homomorphism  $\lambda\colon X\rightarrow \hat{X}$ over $k$
satisfying the equivalent
conditions
\begin{enumerate}[\upshape (a)]
\item
\label{pola1}
$\lambda$ is a symmetric isogeny and the line bundle
$(\id_X,\lambda)^\ast \caP$ is ample;
\item
\label{pola2}
there exists a finite separable field extension $k\subseteq K$ and
an ample line bundle $L$ on $X_K$ such that $\lambda_K=\varphi_L$.
\end{enumerate}
}
\end{definition1}

If $\lambda\colon X\rightarrow \hat{X}$ is a polarization of the abelian variety
$X$, following Mumford \cite[Defns.~7.2, 7.3]{Mu1} define the {\sf degree}
$\deg(\lambda)$ of the polarization $\lambda$ to be the degree of
the isogeny $\lambda$, i.e., $\#\ker(\lambda)$. The degree $\deg(\lambda)$ is
always a square
by the Riemann-Roch theorem: $\deg(\lambda)=d^2$ with
$d=\chi(L)$ if $\lambda_{\overline{k}}=\varphi_L$, see \cite[\S 16]{Mu}.  
It is convenient to define the {\sf reduced degree} $\rdeg(\lambda)$
of the polarization $\lambda$ to be 
\begin{equation}
\label{chorizo}
\rdeg(\lambda)=\sqrt{\deg(\lambda)} .
\end{equation}
A polarization  $\lambda\colon X\rightarrow \hat{X}$ which is an isomorphism is a 
{\sf principal polarization}.  
 If $\lambda\colon X\rightarrow \hat{X}$ is a polarization of the abelian variety
$X$ and $\phi\colon X'\rightarrow X$ is an isogeny,
then 
\begin{equation}
\label{transfer}
\phi^\ast(\lambda)\colonequals \hat{\phi}\circ\lambda\circ\phi\colon X'\rightarrow
\widehat{X'}
\end{equation}
is a polarization of $X'$ with
\begin{equation}
\label{beer}
\deg(\phi^\ast(\lambda))=
\deg(\lambda)\deg(\phi)^2\qquad\text{and}\qquad
\rdeg(\phi^\ast(\lambda))=\rdeg(\lambda)\deg(\phi).
\end{equation}

\begin{definition1}
\label{worst}
{\rm 
Suppose the abelian variety  $X$  over the field $k$
has dimension $g$ and 
polarization $\lambda\colon X\rightarrow \hat{X}$
with kernel $\ker(\lambda)$. The polarization
$\lambda$ is is an {\sf $[\ell]$-polarization}
for a prime $\ell\neq \charr k$ if $\ker(\lambda)\subseteq X[\ell]$.
An $[\ell]$-polarization $\lambda\colon X\rightarrow \hat{X}$ 
has reduced degree $\rdeg (\lambda)=\ell^r$ for
$0\leq r\leq g$.  We say that $\lambda$ is of {\sf type $r$} and
$(X,\lambda)$ is an $[\ell]$-polarized abelian variety of  type $r$.
An $[\ell]$-polarization
of type $0$ is a principal polarization.
If $\lambda\colon X\rightarrow \hat{X}$ is an $[\ell]$-polarization 
of type $r$, then there is a homomorphism
$[\lambda]=[\lambda]_\ell\colon\hat{X}\rightarrow X$ such that
$[\lambda]\circ \lambda$ is multiplication by $\ell$ on $X$.
We will see in Theorem \ref{yam} that $[\lambda]$ is an
$[\ell]$-polarization of type $\hat{r}\colonequals g-r$ on $\hat{X}$.
}
\end{definition1}
\begin{remark1}
{\rm
For an abelian variety $X$ over a field $k$ and $n\in\NN$ prime to 
$\charr k$ there is a 
perfect pairing 
\[
\langle\,\,\, ,\,\,\,\rangle_{n}\colonequals
\langle\,\,\, ,\,\,\,\rangle_{X,n}\colon  X[n]\times \hat{X}[n]\rightarrow
\Bmu_n .
\]
A polarization $\lambda$ on $X$ 
gives rise to the Weil pairing
\begin{equation*}
\label{dented}
\langle\,\,\, ,\,\,\,\rangle_{\lambda,n}\colonequals
\langle\,\,\, ,\,\,\,\rangle_{X,\lambda,n} \colon
X[n]\times X[n]\longrightarrow X[n]\times \hat{X}[n]
\stackrel{\langle \,\,,\,\,\rangle_n}{\longrightarrow}\Bmu_n
\text{ with }
\langle u,v\rangle_{\lambda,n}=\langle u, \lambda(v)\rangle_n.
\end{equation*}
}
\end{remark1}
\begin{proposition}
\label{print}
Let $X$ be an abelian variety over a field $k$ with $\dim X=g$.  Let
$\caP$ be the
Poincar\'{e} line bundle on
$X\times\hat{X}$.
\begin{enumerate}[\upshape (a)]
\item
\label{print0}
Let $\tau\colon X\rightarrow \hat{X}$ be a symmetric isogeny.  The following
are equivalent:
\begin{enumerate}[\upshape (i)]
\item
$\tau$ is a polarization.
\item
$(n\id_X,\tau)^\ast\caP$ is an ample line bundle on $X$ for some $n\in\NN$.
\item
$(n\id_X, \tau)^\ast\caP$ is an ample line bundle on $X$ for all $n\in\NN$.
\item
$n\tau$ is a polarization for some $n\in\NN$.
\item
$n\tau$ is a polarization for all $n\in\NN$.
\end{enumerate}
\item
\label{print2}
Let $\ell\neq \charr k$ be a prime.
If $(X,\lambda)$ is an $[\ell]$-polarized abelian
variety of type $g$, then $\lambda=\ell \lambda '$
for a principal polarization $\lambda'$ of $X$.
\end{enumerate}
\end{proposition}
\begin{proof}
\eqref{print0}:
As in Definition \ref{polar}, the symmetric isogeny
$\eta\colon X\rightarrow \hat{X}$ is a polarization if and only if the line bundle
$(\id_X,\eta)^\ast\caP$ on $X$ is ample. But 
\[
(n\id_X,\eta)^\ast\caP=(\id_X, n\eta)^\ast\caP=(\id_X,\eta)^\ast
\caP^{\otimes n}=((\id_X,\eta)^\ast\caP))^{\otimes n} 
\]
by \cite[Exercise 7.4]{vM}, and so $(n\id_X,\eta)^\ast\caP=
(\id_X, n\eta)^\ast\caP$ is ample
if and only if $(\id_X,\eta)^\ast\caP$ is ample.\\
\eqref{print2}:
If $(X,\lambda)$ is an $[\ell]$-polarized abelian variety of type $g$,
then $\lambda=\ell\lambda'$ for a symmetric isogeny
$\lambda'\colon X\rightarrow \hat{X}$. By \eqref{print0} we have
that $\lambda'$ is a principal polarization of $X$.
\end{proof}

\begin{definition1}
  \label{best}
{\rm
Let $\A=(A,\lambda)$ be an $[\ell]$-polarized $g$-dimensional superspecial
abelian variety over $\oFp$, $p\neq \ell$.  
We denote its $\bFp$-isomorphism class by 
$[\A]$. 
\begin{enumerate}[\upshape (a)]
\item
\label{best2}
For $0\leq r\leq g$, let
 $\sS_{\!g}(\ell,p)_r$ be the set of $\oFp$-isomorphism classes $[\A]$
  of $g$-dimensional $[\ell]$-polarized superspecial abelian varieties
  over $\oFp$ of type $r$. In particular $\sS_{\!g}(p)_0\colonequals
\sS_{\!g}(\ell,p)_0$ is the 
set of $\oFp$-isomorphism classes of principally polarized superspecial
abelian varieties. The sets $\sS_{\!g}(\ell,p)_r$ are finite.
\item
\label{best3}
For $[\A=(A,\lambda)]\in\sS_{\!g}(p)_0$, set
  $\ell \A=
\ell (A,\lambda)=(A, \ell \lambda)$, so $[\ell\A]\in \sS_{\!g}(\ell,p)_g$.  
Suppose $\A'=(A,\lambda')$ with $[\A']\in\sS_{\!g}(\ell,p)_g$.
Then there is a principally
polarized abelian variety $\A=(A,\lambda)$ with  $[\A']=[\ell \A]$
by Proposition \ref{print}\eqref{print2}.
In particular $\sS_{\!g}(\ell,p)_g$ is the set 
of $\oFp$-isomorphism classes of $g$-dimensional 
superspecial abelian varieties over $\oFp$ with $\ell$ times a 
principal polarization. There is a canonical bijection between 
$\sS_{\!g}(p)_0$
and $\sS_{\!g}(\ell,p)_g$ and $\#\sS_{\!g}(\ell, p)_g=h_g(p)$.
\end{enumerate}
}
\end{definition1}
\begin{theorem}
\label{yam}
Suppose $(X,\lambda)$ is a $g$-dimensional $[\ell]$-polarized abelian variety 
of type $r$, $0\leq r\leq g$, over a field $k$. 
Then $[\lambda]=[\lambda]_\ell$ as in Definition \textup{\ref{worst}} is
an $[\ell]$-polarization of $\hat{X}$ of type $\hat{r}\colonequals g-r$.
\end{theorem}
\begin{proof}
Firstly note that $[\lambda]:\hat{X}\rightarrow X$ is symmetric.
Let $\caP$ be the Poincar\'{e} bundle on $X\times \hat{X}$
and let $\cQ$ be the Poincar\'{e} bundle on $\hat{X}\times X$,
where we identify $\hat{\hat{X}}=X$ as in \eqref{lob}.
If $s:X\times \hat{X}\rightarrow \hat{X}\times X$ is the switch factors
map $s(x,y)=(y,x)$, then $s^\ast(\cQ)=\caP$. From this it follows that
\begin{equation}
\label{waiting}
([\lambda],\id_{\hat{X}})^\ast\caP=([\lambda],\id_{\hat{X}})^\ast s^\ast\cQ=
(s\circ([\lambda],\id_{\hat{X}}))^\ast\cQ=(\id_{\hat{X}},[\lambda])^\ast\cQ
\end{equation}
as line bundles on $\hat{X}$.
Now $[\lambda]$ is a polarization of $\hat{X}$ if and only if 
the line bundle $(\id_{\hat{X}},[\lambda])^\ast\cQ$ on 
$\hat{X}$ is ample as in Definition \ref{polar}. But since 
$\lambda:X\rightarrow \hat{X}$ is an isogeny, this is true if and only if 
\[
\lambda^\ast(\id_{\hat{X}},[\lambda])^\ast\cQ=
\lambda^\ast([\lambda],\id_{\hat{X}})^\ast\caP=
(\ell\id_X, \lambda)^\ast\caP
\]
is an ample line bundle on $X$, where we have used
\eqref{waiting}.  But this is true since $\lambda$
is a polarization by Proposition \ref{print}\eqref{print0}.
Since $\deg([\lambda]\circ\lambda)=\ell^{2g}$ and
$\deg(\lambda)=\ell^{2r}$, it follows that
$\deg([\lambda])=\ell^{2\hat{r}}$.  Hence $[\lambda]$ is an 
$[\ell]$-polarization of $\hat{X}$ of type $\hat{r}$.
\end{proof}
\begin{definition1}
\label{hiro}
{\rm
Suppose $\X=(X,\lambda)$ is a $g$-dimensional $[\ell]$-polarized abelian variety
of type $r$, $0\leq r\leq g$. The the {\sf $[\ell]$-dual} of $\X$
is $\hat{\X}=(\hat{X},[\lambda])$ with the $[\ell]$-polarization
$[\lambda]$ on $\hat{X}$ of type $\hat{r}$ as in Proposition \ref{yam}. If 
$[\A]\in
\sS_{\!g}(\ell, p)_r$, then $[\hat{\A}]\in\sS_{\!g}(\ell,p)_{\hat{r}}$. The association
$[\A]\leftrightarrow [\hat{\A}]$ gives a one-to-one correspondence
between $\sS_{\!g}(\ell, p)_r$ and $\sS_{\!g}(\ell, p)_{\hat{r}}$.

The $[\ell]$-dual $\hat{\A}$
of $\A=(A,\lambda)$ with $[\A]\in\sS_{\!g}(p)_0$ is $\ell\A$ as in
Definition \ref{best}\eqref{best3} with
$[\ell\A]\in\sS_{\!g}(\ell, p)_{\hat{0}}=
\sS_{\!g}(\ell, p)_{g}
$.  Likewise the $[\ell]$-dual 
$\widehat{\ell\A}$
of $\ell\A=(A,\ell \lambda)$ with $[\ell\A]\in\sS_{\!g}(\ell,p)_g$ is
$[\A]\in\sS_{\!g}(p)_{\hat{g}}=
\sS_{\!g}(p)_{0}$.
}
\end{definition1}

Now fix $n\in\NN$ prime to $\charr k$.
Let $\lambda$ be a principal polarization 
on the abelian variety $X$ over the
field $k$.  Then $\lambda$ defines an alternating and nondegenerate
Weil pairing on the $n$-torsion $X[n]$ 
of $X$
\begin{equation}
  \label{onion}
  \langle \quad , \quad \rangle_{\lambda,n}\colon X[n]\times X[n]\rightarrow \Bmu_n;
  \end{equation}
$\#X[n]=n^{2g}$.
A subgroup $C\subseteq X[n]$ is {\sf
  $n$-isotropic} if the Weil pairing $ \langle \quad , \quad
\rangle_{\lambda,n}$ is trivial when restricted to $C$.  An
$n$-isotropic subgroup $C$ is {\sf maximal $n$-isotropic} if there is
no $n$-isotropic subgroup of $X$ properly containing $C$.  The order
of a maximal $n$-isotropic subgroup of $X$ is $n^g$.  Put
\begin{equation}
  \label{grunt}
  \Iso_n(\X)=\{\mbox{maximal $n$-isotropic subgroups $C\subseteq X[n]$}\}.
\end{equation}
For a prime $\ell\neq \charr k$ it is known that
\begin{equation}
\label{grant}
\#\Iso_\ell(\X)=N_g(\ell)\colonequals \prod_{k=1}^{g}(\ell^k+1) .
\end{equation}
Define an equivalence relation $\sim$ on $\Iso_n(\X)$ by 
$C\sim C'$ if there exists $\alpha\in\Aut(\X)$ with 
$\alpha(C)=C'$ for $C,\, C'\in\Iso_n(\X)$. Put
\begin{equation}
\label{granted}
\iso_n(\X)=\Iso_n(\X)/\sim
\end{equation}
with $[C]\in\iso_n(\X)$ the equivalence class containing
$C\in\Iso_n(\X)$.

A key fact is that quotienting a principally
polarized abelian variety by a maximal isotropic subgroup
gives an abelian variety which is again principally polarized:

\begin{proposition}
  \label{heard}
\textup{cf. \cite[\S 23, Cor.~to Thm.~2]{Mu} and \cite[p. 36]{Oo}.}
Suppose $\X=(X, \lambda)$ is a principally polarized
abelian variety over an algebraically closed field $k$
 and $C\subseteq X[n]$ with $n$ prime to $\charr k$.
Let $\psi_C:X\rightarrow X/C=:X'$.  Then there is a principal polarization
$\lambda'$ on $X'$ so that $\psi_C^{\ast}(\lambda')=n \lambda$ if and only if
$C\in\Iso_n(\X)$. In this case we write 
$\X'=(X',\lambda')=\X/C$.  Furthermore, if
$[\A]\in\sS_g(p)_0$ and $(n,p)=1$, then $[\A']\in\sS_{\!g}(p)_0$.
\end{proposition}

Recall that we have fixed a supersingular elliptic curve $E=E/\oFp$
with $\cO=\cO_{\HH_p}=\cO_E=\End(E)$;
$\cO$ is a maximal order in the rational quaternion algebra
$\HH_p\cong \End^0(E)\colonequals \End(E)\otimes \Q$.

\begin{remark1}
\label{cabbage}
{\rm
  For $g>1$ polarizations $\lambda$  on $g$-dimensional
superspecial abelian varieties $A=E^g$ in characteristic
$p$ with $\rdeg (\lambda)=d$ as in \eqref{chorizo} are in
one-to-one correspondence with 
$\msH_{g,d}(\cO)$ as in \eqref{hermit}.
Explicitly, let $\lambda_0$ be the product
polarization on $E^g$.  Then the polarization 
$\lambda_H$ corresponding
to $H\in \msH_{g,d}(\cO)$ is
\begin{equation}
\label{pola}
\lambda_H:A\stackrel{H}{\longrightarrow}A
\stackrel{\lambda_{0}}
{\longrightarrow} \hat{A},
\end{equation}
see \cite[Prop.~2.8]{IKO}. Note that for $n\in\NN$ and $H\in\msH_{g,d}(\cO)$
we have $\lambda_{nH}=n\lambda_H$.
}
\end{remark1}
\begin{proposition}
\label{degree} 
Let $\ell\neq p$ be prime.  Let
$A=E^g/\bFp$ with polarizations $\lambda\colonequals\lambda_H$,
$\lambda'\colonequals \lambda_{H'}$ 
corresponding to positive-definite Hermitian matrices 
 $H, H'\in\msH_{g,d}(\cO)$ 
as in \eqref{pola}.
Let
$\phi:A\rightarrow A$ be an isogeny of degree $\ell^{gm}$
given by $M\in \Mat_{g\times g}(\cO)$.
Then $\phi^\ast (\lambda')=\ell^m \lambda$ if and only if $M^\dagger H'M=\ell^m H$.
\end{proposition}
\begin{proof}
By \eqref{transfer} $$\phi^\ast(\lambda')=
\hat{\phi}\circ\lambda'\circ\phi ,$$ which by \eqref{pola}
equals $$\hat{\phi}\circ\lambda_{0}\circ H'\circ\phi=
\widehat{M}\lambda_{0}H'M=
\lambda_{0}\lambda_{0}^{-1}\widehat{M}\lambda_{0}H'M.$$ Now
$\lambda_{0}^{-1}\widehat{M}\lambda_{0}$ is the Rosati
anti-involution applied to $M$ by definition which equals $M^\dagger$
in the product polarization case.  Thus we have $$\phi^\ast(\lambda')=
\lambda_{0}M^{\dagger}HM= \lambda_{M^{\dagger}HM}.$$ Hence,
$\phi^*(\lambda')=\lambda_{M^{\dagger}HM}$ and since $\ell^m \lambda=\ell^m\lambda_H=\lambda_{\ell^m H}$, we are
done by Remark \ref{cabbage}.
\end{proof}

This allows us to describe the set $\sS_{\!g}(p)_0$ for $g>1$ following \cite{IKO}.  
\begin{proposition}
  \label{only}
  If $g>1$ then the map
\[
\overline{\msH}_{\!\!g,1}(\cO)\ni[H]\mapsto [\A(H)],\mbox{ where }\A(H)=(A, \lambda_H),
\]
with $\cO=\cO_{\HH_p}$
is a bijection between $\omsH_{\!\!g,1}(\cO)$ defined in \eqref{class} and $\sS_{\!g}(p)_0$.  
\end{proposition}

We thus obtain the following description of $\sS_{\!g}(p)_0$.

\begin{theorem}
  \label{gar}
\textup{(Ibukiyama/Katsura/Oort, Serre) }
There are one-to-one correspondences $\leftrightarrow$
with $\cO=\cO_{\HH_p}$ \textup{:}
\begin{enumerate}[\upshape (a)]
\item
  \label{gar1}
For $g\geq 1$,
\[
\sSgpz\longleftrightarrow \msP_g(\HH_p)= \GU_g(\HH_p)\backslash 
\GU_g(\widehat{\HH}_p)/\GU_g(\cO_{\widehat{\HH}_p}).
\]
\item
  \label{gar2}
For $g>1$,
\[
\sSgpz\longleftrightarrow \msP_g(\HH_p)=\GU_g(\HH_p)\backslash 
\GU_g(\widehat{\HH}_p)/\GU_g(\cO_{\widehat{\HH}_p})
\longleftrightarrow  \overline{\msH}_{g,1}(\cO) =\msH_{g,1}(\cO)/\GL_g(\cO) ,
\]
where the second one-to-one correspondence is Theorem \textup{\ref{t2}}.
\end{enumerate}
\end{theorem}

\begin{theorem}
\label{elect}
Let $\A=(A,\lambda)$ and $\A'=(A',\lambda')$ with $[\A], [\A']\in
\sS_{\!g}(p)_0$ for  $g>1$ and let $\ell\neq p$ be a prime.
Suppose $\psi:A'\rightarrow A$ is  an isogeny such that 
$\psi^\ast(\lambda)=\ell^m \lambda'$ for $m\geq 1$. Then there exist
principally polarized superspecial abelian varieties
\begin{align*}
&(A_1, \lambda_1)=\A_1=\A'=(A',\lambda'),\, \A_2 =(A_2, \lambda_2),\,\ldots ,
\A_m=(A_m,\lambda_m), \\
&\quad\quad\quad\qquad\,\, (A_{m+1}, \lambda_{m+1})=\A_{m+1}=\A=(A,\lambda)
\end{align*}
 with $(\ell)^g$-isogenies $\psi_i\colon A_{i}\rightarrow A_{i+1}$
such that $\psi_i^\ast(\lambda_{i+1})=\ell \lambda_i$ for $1\leq i\leq m$ and
$\psi=\psi_{m}\circ\psi_{m-1}\circ \cdots \circ\psi_{1}$\textup{:}
\[
\psi:A'=A_1\stackrel{\psi_1}{\longrightarrow}A_2\stackrel{\psi_2}
{\longrightarrow}\,\cdots\, \stackrel{\psi_{m-1}}{\longrightarrow}
A_m\stackrel{\psi_m}{\longrightarrow}A_{m+1}=A.
\]
\end{theorem}
Theorem \ref{elect} will follow from the purely algebraic Theorem
\ref{elect1} below.  
\begin{theorem}
\label{elect1}
Let $V$ be a free $\Z/\ell^n\Z$-module of rank $2g$ with a
nondegenerate symplectic pairing 
\[
\langle \,\,\, ,\,\,\,\rangle_V\colon V\times V\longrightarrow \Q/\Z.
\]
Note that there is an induced nondegenerate symplectic pairing
on the $\ell$-torsion $V[\ell]\subseteq V$
\[
\langle \,\,\, ,\,\,\,\rangle_{V[\ell]}\colon V[\ell]\times V[\ell]\longrightarrow
\Q/\Z \quad\text{by}\quad
\langle \bullet , \bullet\rangle_{V[\ell]}=\langle(1/\ell^{n-1})
\bullet, \bullet\rangle_V .
\]
Let $M\subseteq V$ be a maximal isotropic subspace.  Then there exists
$G\subseteq M[\ell]$ such that $G\subseteq V[\ell]$ is maximal isotropic
with respect to $\langle \,\,\, ,\,\,\,\rangle_{V[\ell]}$.
\end{theorem}
\begin{proof}
The proof is by induction on $g$.  If $g=1$, let $G$ be any line
in $M[\ell]$. Suppose the statement is true for $g-1$. \\
{\em Case \textup{1.} } $V[\ell]\subseteq M$.  In this case let $G$ be any
maximal isotropic subgroup of $V[\ell]$.\\
{\em Case \textup{2.} } $V[\ell]\not\subseteq M$.  In this case there exists
$N\subseteq V$, $N\cong (\Z/\ell^n\Z)^{2g-1}$, such that $M\subseteq N$.
To see this note that $M$ has at most $2g-1$ generators, lift them
arbitrarily to $\ell^n$-torsion to define $N$.

Note that $N^\perp\cong \Z/\ell^n\Z$ and $\langle N^\perp, M\rangle_V=0$
since $M\subseteq N$.  So $N^\perp\subseteq M$ by maximality.  Apply
the induction hypothesis to $M/N^\perp\subseteq N/N^\perp$:
$N/N^\perp$ is rank $2g-2$ over $\Z/\ell^n\Z$ with a nondegenerate
symplectic pairing induced by $\langle \,\,\, , \,\,\,\rangle_V$.
Also $M/N^\perp$ is isotropic; it is maximal isotropic since if it were
contained in a bigger isotropic subgroup pulling back would contradict the
maximality of $M$.  Hence by the induction hypothesis there exists 
$\tilde{G}\subseteq M$
with $N^\perp\subseteq \tilde{G}$
such that $\tilde{G}/N^\perp\subseteq
(M/N^\perp)[\ell]$ is maximal isotropic.  Now take $G=\tilde{G}[\ell]$.
\end{proof}
{\em Proof of Theorem \textup{\ref{elect}}. }The proof is by induction
on $m$. For $m=1$ the statement follows from Proposition \ref{heard}.
 Suppose the statement is true for $m$ and consider
$\A'=(A,\lambda')\colonequals (A_1,\lambda_1)$, 
$\A=(A,\lambda)\colonequals \A_{m+2}=(A_{m+2}, \lambda_{m+2})$ 
with $[\A],[\A']\in\sS_g(p)_0$ for $g>1$ and an
isogeny $\psi:A'\rightarrow A$ such that $\psi^\ast(\lambda)=\ell^{m+1}\lambda'$
for $m\geq 1$.  By Proposition \ref{heard}, the kernel 
$C\subseteq A'[\ell^{m+1}]$  of $\psi$ is
a maximal $\ell^{m+1}$-isotropic subgroup.  Now apply Theorem \ref{elect1}
to the free $\Z/\ell^{m+1}\Z$-module $A'[\ell^{m+1}]$ with the 
nondegenerate symplectic pairing $\langle\,\,\, ,\,\,\,\rangle_{\lambda',\ell^{m+1}}$.
This shows there exists an $\ell$-maximal isotropic subgroup 
$G\subseteq C[\ell]\subseteq A'[\ell]$ with respect to the nondegenerate
symplectic pairing
$\langle \,\,\, ,\,\,\,\rangle_{\lambda',\ell}$. 

Let  $\A_2$ be
the principally polarized abelian variety $\A_2=(A_2,\lambda_2)\colonequals
\A'/G$ with isogeny $\psi_1=\psi_G\colon A_1\colonequals A'\rightarrow A_2$.
Since $G\subseteq C$, the isogeny $\psi\colon A'\rightarrow A$ 
factors as
\[
\psi\colon A_1=A'\stackrel{\psi_1}{\longrightarrow}A_2\stackrel{\psi'}
{\longrightarrow}A_{m+2}=A.
\]
Note that both $\ell^m \lambda_2$ and $\psi^{\prime \ast}(\lambda_{m+2})$ are
polarizations on $A_2$ which pull back under $\psi_1$ to 
$\ell^{m+1}\lambda_1$.  Since the N\'{e}ron-Severi group of an abelian
variety is torsion-free, this implies that
$\psi^{\prime \ast}(\lambda_{m+2})=\ell^m \lambda_2$. Applying the
induction hypothesis
to $\psi':A_2\rightarrow A_{m+2}$ with $\psi^{\prime \ast}(\lambda_{m+2})=
\ell^m \lambda_2$ now concludes the proof. \qed

\section{The big, little, and enhanced isogeny graphs}
\label{ble}

\subsection{The big isogeny graph \texorpdfstring{\except{toc}{\boldmath{$\GR_{\!g}(\ell,p)$}}\for{toc}{$\GR_{\!g}(\ell,p)$}}{Gr\unichar{"005F}g(\unichar{"2113},p)}}
\label{gr1}

The {\sf big} $(\ell)^g$-isogeny graph $\GR=\GR_{\!g}(\ell,p)$ 
(often called simply ``the isogeny graph'')
is the directed
graph with vertices $\Ver(\GR)= \sS_{\!g}(p)_0=
\{[\A_1=(A, \lambda_1)],\ldots , [\A_h=(A, \lambda_h)]$\} with
$\#\Ver(\GR)=h=h_g(p)$ and $A=E^g/\oFp$.  Its edges are
\begin{equation}
  \label{but}
\Ed(\GR)_{ij}=\{C\in\Iso_\ell(\A_i)\mid [\A_i/C]=
[\A_j]\}
\end{equation}
with $\Iso_\ell(\A)$ as in \eqref{grunt}.
A useful reformulation of \eqref{but} is the following:
Set
\begin{align*}
\Hom(\A_i,\A_j)_\ell &  =\{\text{isogenies $\phi\colon \A_i\rightarrow
  \A_j$ of degree $\ell^g$ such that $\phi^\ast(\lambda_j)=\ell\lambda_i$}\}
\text{ and}\\
\Aut(\A_j) &= \{\text{automorphisms $\psi:A_i\rightarrow A_j$}\}.
\end{align*}
Define the equivalence relation $\sim_{b}$ on $\Hom(\A_i,\A_j)$
by $\phi\sim_b\phi'$ if there is an automorphism $\alpha$
of $\A_j$ such that $\phi'=\alpha\circ \phi$ and set
$\overline{\Hom}(\A_i,\A_j)_\ell=\Hom(\A_i,\A_j)_\ell/\mathord{\sim}_b$.
Then by Proposition \ref{heard} we have
\begin{align}
\nonumber
  \Ed(\GR)_{ij} & =\overline{\Hom}(\A_i,\A_j)_\ell \quad\text{and}\\
\#\Ed(\GR)_{ij} & =\frac
  {\#\Hom(\A_i,\A_j)_\ell}{\#\Aut{(\A_j)}}.\label{but1}
  \end{align}

  We have
$\sum_{j=1}^h\#\Ed(\GR)_{ij}= N_g(\ell)=\prod_{k=1}^g(\ell^k+1)$;
see \eqref{grant}.

\begin{theorem}
\label{adbig}
Let $\cO\subseteq \HH_p$ be the maximal order $\End(E)$ with
$B_g(\ell)$ the Brandt matrix for the maximal order $\cO$. Then
\begin{enumerate}[\upshape (a)]
\item
\label{adbig1}
$\GR_{\!g}(\ell, p)=\BR_{\!g}(\ell,\cO)$,
\item
\label{adbig2}
$\Ad(\GR_{\!g}(\ell, p))=B_{g}(\ell)$, and
\item
  \label{adbig3}
  the big isogeny graph $\GR_{\!g}(\ell,p)$ is regular
  of degree $\N_g(\ell)=\prod_{k=1}^{g}(\ell^k+1)$.
\end{enumerate}
\end{theorem}
\begin{proof}
  The case $g=1$ is classical and well-known: combine Remark \ref{Brandt well}
  with the quaternionic-ideal description of isogenies of supersingular elliptic curves as 
  in, for example, \cite[\S 2]{Gr}.

  So suppose $g>1$. \eqref{adbig1}:
  With $\cO=\cO_{\H_p}$, we have
  \[
  \Ver(\GR_{\!g}(\ell, p))=\Ver(\BR_{\!g}(\ell,\cO))= \sS_{\!g}(p)_0
  \leftrightarrow \cP_{\!g}(\cO)\leftrightarrow
  \overline{\msH}_{g,1}(\cO) \colonequals \msH_{g,1}(\cO)/\SL_g(\cO)
    \]
    using Theorem \ref{gar}\eqref{gar2}
    and the definitions in Sections \ref{gr1} and \ref{rent}.
    With $h=h_g$ and $ \overline{\msH}_{g,1}(\cO)=\{[H_1],
    \ldots , [H_h]\}$ for $H_i\in\msH_{g,1}$, $1\leq i\leq h$,
    as in \eqref{bid}
    we have  $\sS_{\!g}(p)_0=\{[\A_1\colonequals (A,\lambda_{H_1})],
    \ldots, [\A_h\colonequals
      (A, \lambda_{H_h})]\}$ for $A=E^g/\oFp$ by Proposition \ref{only}.
    But now using the notation of Definition \ref{groan}
    and Theorem \ref{Brandt2} we have
    \begin{align}
 \nonumber     \Ed(\GR_g(\ell, p))_{ij} &=\overline{\Hom}(\A_i,\A_j)_\ell\\
 \nonumber      & = \mathbf{U}_\ell(H_i, H_j)^{\text{\rm big}}
 \text{ by Prop.~\ref{degree}}\\
 \label{rated}       &= \Ed(\BR_g(\ell, p))\text{ by Thm.~\ref{Brandt2}}
          \end{align}
Since the edges and vertices of $\GR_{\!g}(\ell, p)$
   and $\BR_{\!g}(\ell,\cO)$ correspond,
   we have $\GR_{\!g}(\ell, p)\cong\BR_{\!g}(\ell,\cO)$.\\
   \eqref{adbig2}: This follows immediately from \eqref{adbig1}
   using \eqref{grinder}.\\
   \eqref{adbig3}:  This follows from \eqref{adbig2}
   by Theorem \ref{hen} \eqref{hen1}.
\end{proof}

Taking the dual isogeny does {\em not} give a well-defined involution on
$\Ed(\GR)$, so the big isogeny graph $\GR_{\!g}(\ell,p)$ is not a graph with opposites.

\subsection{The little isogeny graph \texorpdfstring{\except{toc}{\boldmath{$\gr_{\!g}(\ell,p)$}}\for{toc}{$\gr_{\!g}(\ell,p)$}}{gr\unichar{"005F}g(\unichar{"2113},p)}}
\label{gr2}

The {\sf little} $(\ell)^g$-isogeny graph $\gr=\gr_{\!g}(\ell,p)$ has vertices
$\Ver(\gr)=\sS_{\!g}(p)_0$, so the big graph $\GR$ and the little graph
$\gr$ have the same vertices.  The edges of
$\gr$ are
\begin{equation}
  \label{soon}
\Ed(\gr)_{ij}=\{[C]\in\iso_\ell(\A_i)\mid 
[\A_i/C]=[\A_j]\}
\end{equation}
with $\iso_\ell(\A)=\Iso_\ell(\A)/\sim$ as in \eqref{granted}.
Given an edge $e\in\Ed(\gr)_{ij}$ with
$e=[C]\in \iso_\ell(\A_i)$ we define its
opposite edge $\overline{e}\in\Ed(\gr)_{ji}$   by
$\overline{e}=[\widehat{C}]\in\iso_\ell(\A_j)$ with $\widehat{C}$
the kernel 
of the dual isogeny $\A_j\rightarrow \A_i$.  Note
that this dual is only well-defined up to $\sim$; thus, it is an
operation on $\gr$ (but not $\GR$).  The little graph
$\gr$ is therefore a graph with opposites.  In general $\gr$ is a graph
with half-edges.

Again we can reformulate \eqref{soon} in terms of isogenies.
Recall from Section \ref{gr1} that
\begin{equation}
  \label{soon1}
  \Ed(\GR)_{ij}=\overline{\Hom}(\A_i,\A_j)_\ell;
\end{equation}
an isogeny $\phi\in \Hom(\A_i,\A_j)_\ell$ defines a class
$[\phi]\in\overline{\Hom}(\A_i,\A_j)$. Define an equivalence
relation $\sim_l$ on $\overline{\Hom}(\A_i,\A_j)_\ell$ by
$[\phi]\sim_l [\phi']$ if $[\phi']=[\phi\circ \beta]$ for
$\beta\in\Aut(\A_i)$ and set
\[
\overline{\overline{\Hom}}(\A_i,\A_j)_\ell
=\overline{\Hom}(\A_i,\A_j)_\ell/\sim_l.
\]
Then
\begin{equation}
  \label{soon2}
  \Ed(\gr)_{ij}=\overline{\overline{\Hom}}(\A_i,\A_j)_\ell.
  \end{equation}

\begin{definition1}
\label{weight}
{\rm
Define a weight function $\w$ on the small graph $\gr=\gr_{\!g}(\ell,p)$ by
$\w([\A])=\#\Aut(\A)$ and $\w(C)=\#\Aut(A, \lambda, C)$ for 
the vertex corresponding to $[\A=(A,\lambda)]\in\sS(g,p)$ and
the edge corresponding to $[C]\in\iso_\ell(\A)$, respectively.  Then
$\gr$ is a weighted graph with half-edges.
}
\end{definition1}
\begin{theorem}
\label{smad}
Let $\cO=\End(E)\subseteq \HH_p$ and let 
$B_g(\ell)$ be the Brandt matrix for $\cO$.  Then
\begin{enumerate}[\upshape (a)]
\item
\label{smad1}
$\gr_{\!g}(\ell, p)=\br_{\!g}(\ell, \cO)$ and 
\item
\label{smad2}
$\Ad_{\w}(\gr_{\!g}(\ell,p))=\Ad(\GR_{\!g}(\ell,p))=B_g(\ell)$.
\end{enumerate}
\end{theorem}
\begin{proof}
  \eqref{smad1}: Again the case $g=1$ is classical and follows
  from Remark \ref{Brandt well}.

  So suppose $g>1$. We have
  \[
  \Ver(\gr_g(\ell, p))=\Ver(\GR_g(\ell, p))=\Ver(\BR_g(\ell, p))=
  \Ver(\br_g(\ell,p))
  \]
  from the proof of Theorem \ref{adbig}\eqref{adbig1}.

  We have
      \begin{align}
        \nonumber     \Ed(\gr_g(\ell, p))_{ij} &=
        \overline{\overline{\Hom}}(\A_i,\A_j)_\ell\text{ by \eqref{soon2}}\\
          \nonumber      & = \mathbf{U}_\ell(H_i, H_j)^{\text{\rm little}}
          \text{ by Prop.~\ref{degree}}\\
 \label{rated}       &= \Ed(\br_g(\ell, p))_{ij}\text{ by Thm.~\ref{Brandt2}}.
          \end{align}
Since the edges and vertices of $\gr_{\!g}(\ell, p)$
   and $\br_{\!g}(\ell,\cO)$ correspond,
   we have $\gr_{\!g}(\ell, p)\cong\br_{\!g}(\ell,\cO)$.\\
   \eqref{smad2}: The equality of $\Ad_{\w}(\gr_g(\ell,p))$ and
   $\Ad(\GR_g(\ell,p))$ follows since each edge $[C]\in\Ed(\gr_g(\ell,p))$
   corresponds to a number of edges of $\Ed(\GR_g(\ell,p))_{ij}$
   equal to the size of the orbit of $C$ under $\Aut(\A_i)$ which equals
   \[
   \frac{\#\Aut(\A_i)}{\#\Aut(A_i,\lambda_i,C)}=\frac{\w([\A_i])}{\w([C])}.
   \]
   Now apply Theorem \ref{adbig}\eqref{adbig2}.
\end{proof}

\subsection{The enhanced isogeny graph \texorpdfstring{\except{toc}{\boldmath{$\wgr_{\!g}(\ell,p)$}}\for{toc}{$\wgr_{\!g}(\ell,p)$}}{tilde gr\unichar{"005F}g(\unichar{"2113},p)}}
\label{gr3}
In the notation of Definition \ref{best},
put $h=h_g(p)$ and 
\begin{align*}
\sS_{\!g}(p)_0&=\{[\A_1],\ldots, [\A_h]\}=\{v_1,\ldots , v_h\},\\
\sS_{\!g}(p)_g &=\{[\ell\A_1], \dots, [\ell\A_h]\}=\{v_{h+1},\ldots , v_{2h}\}.
\end{align*}
The {\sf enhanced} $(\ell)^g$-isogeny graph $\wgr=\wgr_{\!g}(\ell,p)$ has
vertices 
\[
\Ver(\wgr)=\sS_{\!g}(p)_0\coprod \sS_{\!g}(p)_g=\{v_1,\ldots, v_h\}\coprod \{v_{h+1},\ldots , v_{2h}\}.
\]
Polarizations of type $g$ are just $\ell$ times a principal
polarization, and thus there is a natural bijection
between $\sS_{\!g}(p)_0$ and $\sS_{\!g}(p)_g$. Nevertheless, they
correspond to distinct vertices of $\wgr$.  For 
$\sS_{\!g}(p)_g\ni[\hat{\A}_i]=[\ell\A_i]=v_{h+i}\in\Ver(\wgr)$ and
$\sS_{\!g}(p)_0\ni[\A_j]=v_j\in\Ver(\wgr)$, the edges of $\wgr$ from
$v_{h+i}$ to $v_j$ are
\[
\Ed(\wgr)_{h+i,j}=\{[C]\in \iso_\ell(\A_i)\mid [\A_i/C]=[\A_j]\}
\]
with notation as in \eqref{granted}.
For $\sS_{\!g}(p)_0\ni[\A_i]=v_i\in\Ver(\wgr)$ and $\sS_{\!g}(p)_g
\ni[\hat{\A}_j]=[\ell\A_j]=v_{h+j}\in\Ver(\wgr)$, the edges of $\wgr$ from
$v_i\in\Ver(\wgr)$ to $v_{h+j}\in\Ver(\wgr)$ are
\[
\Ed(\wgr)_{i, h+j}=\{[\hat{C}]\in \iso_\ell(\hat{\A}_i)\mid 
[\hat{\A}_i/\hat{C}]=[\hat{\A}_j]\}
\]
with $\hat{\A}$ denoting the $[\ell]$-dual of $\A$
as in Definition \ref{hiro}.  
In case $1\leq i,j\leq h$ or $h+1\leq i,j\leq 2h$, $\Ed(\wgr)_{ij}=\emptyset$.

The enhanced isogeny graph $\wgr$ is a graph with opposites: If 
$e\in\Ed(\wgr)_{ij}$ the opposite edge $\overline{e}\in\Ed(\wgr)_{ji}$
is the equivalence class of the dual isogeny.
We never have $\overline{e}=e$, so $\wgr$ is a graph without
half-edges.  The graph $\wgr$ is a graph with weights: define $\w$
as the order of the automorphism group as for $\gr$.

\begin{theorem}
\label{dealing}
\begin{enumerate}[\upshape (a)]
\item
\label{dealing1}
The enhanced isogeny graph $\wgr=\wgr_{\!g}(\ell, p)$ is the bipartite double
cover of the little isogeny graph $\gr=\gr_{\!g}(\ell,p)$ with inherited weights.
\item
\label{dealing2}
Let $A=\Ad(\gr)$ and $A_{\w}=\Ad_{\w}(\gr) =\Ad(\GR_{\!g}(\ell,p))$.  Then
\[
\Ad(\wgr)=\begin{bmatrix}0 & A\\A & 0\end{bmatrix}
\quad\text{and}\quad \Ad_{\w}(\wgr)=
\begin{bmatrix} 0 & A_{\w}\\ A_{\w} &0\end{bmatrix}=
\begin{bmatrix} 0 & B_g(\ell)\\B_g(\ell) & 0\end{bmatrix}.
\]
\end{enumerate}
\end{theorem}
\begin{proof}
\eqref{dealing1}:
Let $\iota:\wgr\rightarrow \wgr$ be the involution defined 
on vertices by $\iota([\A])=[\hat{\A}]$ and on edges 
such that if $e\in\Ed(\wgr)_{ij}$ corresponds to the class $[C]$,
then $\iota(e)\in\Ed(\wgr)_{i+h,j+h}$ (where the indices are added mod
$2h$) also corresponds to the class $[C]$.
Then $\iota$ fixes
no vertices and no edges of $\wgr$ and $\wgr/\iota=\gr$.
Thus the enhanced graph $\wgr$ is the bipartite double
cover of the little graph $\gr$. \\
\eqref{dealing2}:
Given \eqref{dealing1}, the adjacency matrices for $\wgr$ now follow from
Theorems \ref{adbig} and \ref{smad}.
\end{proof}

\section{Connectedness results for isogeny graphs}
\label{co}

\subsection{Connectedness for \texorpdfstring{\except{toc}{\boldmath{$g=1$}}\for{toc}{$g=1$}}{g=1}: supersingular elliptic curves}
\label{el}

It is  well known  that the $\ell$-isogeny graph for
supersingular elliptic curves in characteristic $p$
is connected.  A standard proof of this result relies on the fact that
integral primitive quaternary quadratic forms represent all sufficiently
large integers.  There is another proof by Serre~\cite[p.~223]{M}
using that the space
of Eisenstein series of weight $2$ for the congruence subgroup
$\Gamma_{0}(p)$ is $1$-dimensional. In this section we give the
proof using Theorem \ref{sun} on strong approximation. 
As a byproduct we get
that $\GR_1(\ell, p)$ and $\gr_1(\ell, p)$ are not
bipartite.  This in turn enables us to conclude that the
enhanced isogeny graph $\wgr_1(\ell, p)$ is connected.

Let $E/\bFp$, $E'/\bFp$ be supersingular elliptic curves
with $\cO=\cO_E=\End(E)$ and $\cO'=\cO_{E'}=\End(E')$ maximal
orders in $\HH_p$.  Then $\Hom(E,E')$ is an ideal
in $\HH_p$ with left order $\cO'$ and right order $\cO$.  
\begin{lemma}
\label{birdy}
If $\psi\in\Hom(E',E)$ has degree $\deg \psi=x\neq 0$, then the right
$\cO$-ideal 
\[
I=\{\psi\circ\phi\mid \phi\in\Hom(E,E')\}\subseteq \cO
\]
has reduced norm
$x$.
\end{lemma}
\begin{proof}
We begin with the case when $\psi$ is separable.  Then we
have $$I=\{\alpha\in\End(E)=\cO \mid \widehat{\alpha}(\ker
\hat{\psi})=0\}.$$ Thus for each prime power $\ell^k\Vert x$, $I\otimes \Z_\ell$
is of index $\ell^{2k}$ in $\cO\otimes \Z_\ell$.
Combining these together we see that $I$ has index
$x^2$ in $\cO$ and hence has reduced norm $x$.

Now suppose $\psi$ is inseparable.  Let $\psi'$ be a separable map from
$E'$ to $E$ of degree $x'$.  Let $\beta=\psi\circ\widehat{\psi}'$.
Let
\[
I'=\{\psi'\circ\phi\mid \phi\in\Hom(E,E')\}\subseteq \cO.
\]
Then by the above case $I'$ has reduced norm $x'$.  Also
$I=\frac\beta{x'}I'$ and taking norms of both sides we get that the
reduced norm of $I$ is $x$.
\end{proof}

\begin{theorem}
\label{g1}
Let $\ell\ne p$ be prime.  
\begin{enumerate}[\upshape (a)]
\item
\label{gl1}
The big isogeny graph $\GR_1(\ell ,p)$ and the little
isogeny graph $\gr_1(\ell,p)$   for supersingular elliptic curves  
are connected.
\item
\label{gl2}
The graphs $\GR_1(\ell, p)$ and $\gr_1(\ell, p)$
are not bipartite,
i.e., given any two supersingular elliptic curves $E$
and $E'$ in characteristic $p$, there exists an isogeny $\phi:
E\to E'$ such that the degree of $\phi$ is an even power of $\ell$.
\item
\label{gl3}
The enhanced isogeny graph $\wgr_1(\ell, p)$ is connected.
\end{enumerate}
\end{theorem}
\begin{proof}
  (\ref{gl1}, \ref{gl2}): Let $E=E/\oFp$ and $E'=E'/\oFp$ be any
  two supersingular elliptic curves.
By Tate's theorem $E$ and $E'$ are isogenous.  Hence there exists
an isogeny $\psi\in\Hom(E',E)$ with some degree $x\ne0$.  Consider
the right ideal $I\subset\cO_E$ defined by $I=\{\psi\circ\phi \mid
\phi\in\Hom(E,E')\}$; $I$ has reduced norm $x$ by 
Lemma \ref{birdy}.  Let $\alpha\in\H_p$ be an element of
norm $x$; such an $\alpha$ exists by the Hasse-Minkowski theorem.
Then the fractional right ideal $I_1=\alpha^{-1}I$ has norm $1$.

Now by Lemma \ref{l0}, there exists an element $\beta\in
I_1\otimes\Z[1/\ell]$ of norm $1$.  Let $\ell^n$ be a sufficiently
high power of $\ell$ so that $\ell^n\beta\in I_1$.  Then
$\alpha\ell^n\beta\in I$ and thus is equal to $\psi\circ\phi$ for some
$\phi\in\Hom(E,E')$.

Taking the equation $\alpha\ell^n\beta=\psi\circ \phi$ and computing
norms/degrees, we obtain 
\[
\Norm_{\HH_p/\Q}(\alpha)\ell^{2n}\Norm_{\HH_p/\Q}(\beta)=
\deg(\psi)\deg(\phi).
\]
Since $\Norm_{\HH_p/\Q}(\alpha)=\deg(\psi)=x$ and $\Norm_{\HH_p/\Q}(\beta)=1$,
we see that the degree of $\phi$
is $\ell^{2n}$. Hence $\GR_1(\ell,p)$ and $\gr_1(\ell, p)$ are
connected and not bipartite.\\
\eqref{gl3}:
Since $\gr_1(\ell,p)$ is connected and not bipartite, its bipartite
double cover $\wgr_1(\ell, p)$ (see Theorem \ref{dealing}\eqref
{dealing1}) is connected.
\end{proof}

\subsection{Connectedness for \texorpdfstring{\except{toc}{\boldmath{$g>1$}}\for{toc}{$g>1$}}{g>1}}
\label{hi}
We now consider the higher-dimensional case; henceforth
suppose $g>1$. Here we deduce the connectedness of the
isogeny graphs from strong approximation for the quaternionic
unitary group.  Strong
approximation in this context has previously been applied
 to questions of moduli of abelian varieties
in characteristic $p$: applications to Hecke orbits are in 
Chai/Oort \cite[Prop.~4.3]{CO}
and applications to the geometry
of stratifications are in  Ekedahl/Oort \cite[\S 7]{EO}; see also
Chai \cite[Prop.~1]{Ch}. In particular, Theorem \ref{g>1} below
should be compared with Ekedahl/Oort's version of strong approximation
in
\cite[Lemma 7.9]{EO}. Combining strong approximation
with Proposition \ref{degree} 
and Theorem \ref{elect} shows
 that the isogeny graphs $\GR_{\!g}(\ell,p)$
and $\gr_{\!g}(\ell,p)$ are connected.
Our strong approximation
argument further implies that  $\GR_{\!g}(\ell, p)$ and
$\gr_{\!g}(\ell ,p)$ are not bipartite.  This in turn is used to
show that the enhanced isogeny graph $\wgr_{\!g}(\ell, p)$ is connected
--- analogously to the $g=1$ argument of Theorem \ref{g1}.
Note that \cite[\S 7]{EO} treats inseparable isogenies
of superspecial abelian varieties which we do not consider here.

Let $\H/\Q$ be an arbitrary rational definite quaternion algebra with
maximal order $\cO_\H$.

\begin{theorem}
\label{g>1}
\text{\rm (cf. \cite[Lemma 7.9]{EO}) }
Let $\ell$ be a prime unramified in $\H$. Then given any
two positive-definite Hermitian matrices $H, H'\in\MM (\cO_{\HH})$
of reduced norm $1$, there exists
a matrix $M\in\MM (\cO_{\HH})$ such that
\begin{equation}
\label{oyster}
M^\dagger HM=\ell^{2n}H'
\end{equation}
for some positive integer $n$.
\end{theorem}

\begin{proof}
Let
$M_0\in\MM(\HH)$ satisfy $M_0^\dagger M_0=H$; such an $M_0$
exists by Lemma \ref{l2}.  By the same lemma, we can assume that
$H^\prime=I$.

We are now ready to apply the strong approximation Theorem
\ref{sun}.  Let $\Aa$ be the ad\`{e}les of $\Q$ and $G$
be the quaternionic unitary group 
\[
G=\rU_g(\HH)=\{M\in\MM (\H) \mid M^\dagger
M=\Id_{g\times g}\}.  
\]
The quaternionic unitary group $G$ is the compact real form
of $\Sp_{2g}$, so is simple.

Let $S=\{\ell,\infty\}$
and set 
\begin{equation}
\label{read}
U=\{M\in G(\Aa)\mid (M_0^{-1}M)_q\in\MM(\cO_{\HH}\otimes \Z_q)\text{ for }
q\neq \ell\}.
\end{equation}
By Lemma \ref{l2} there exists $N_q\in\MM(\cO_{\HH}\otimes \Z_{(q)})$
such that $H=N_q^\dagger N_q$.  Then 
\[
M_0N_q^{-1}\in\rU_g(\cO_{\HH}\otimes
\Q_q)=G(\Q_q).
\]
The set $U\subseteq G(\Aa)$ in \eqref{read} is open and nonempty
since $(M_0N_q^{-1}, M_p)_{q\notin S,\, p\in S}\in U$ for $M_p$ arbitrary.

Hence by strong approximation (Theorem \ref{sun}) there exists
$M'\in\MM(\HH)$ such that $M'^\dagger M'=\Id_{g\times g}$ and
$M_0^{-1}M'\in\MM(\cO_{\HH}[1/\ell])$.  Let $\ell^n$ be a sufficiently
high power of $\ell$ such that $\ell^n M_0^{-1}M'\in\MM(\cO_{\HH})$.
Let $M=\ell^nM_0^{-1}M'$.  Then
\begin{equation*}
  M^\dagger HM=M^\dagger M_0^\dagger M_0M=\ell^{2n}M'^\dagger M'
  =\ell^{2n}\Id_{g\times g}.\
  \end{equation*}
\end{proof}

\begin{theorem}
\label{fun}
Let $\ell\ne p$ be prime, $g>1$, $A=E^g$, and $\cO=\cO_E=\End(E)$.
\begin{enumerate}[\upshape (a)]
\item
\label{fun1}
The big isogeny graph $\GR_{\!g}(\ell,p)$ and the little isogeny
graph $\gr_{\!g}(\ell,p)$ are connected.
\item
\label{fun2}
The graphs $\GR_{\!g}(\ell, p)$ and $\gr_{\!g}(\ell, p)$ are not bipartite,
i.e., given any
two principal polarizations of $A$, $\lambda_H$ and $\lambda_{H'}$ with $H,H'\in\SL_g(\cO)$
positive-definite Hermitian matrices, there exists
a path on each graph from the vertex $[\A'=(A, \lambda_{H'})]$ to the
vertex $[\A=(A,\lambda_H)]$ of even length.
\item
\label{fun3}
The enhanced isogeny graph $\wgr_g(\ell, p)$ is connected.
\end{enumerate}
\end{theorem}
\begin{proof}
Put  $\GR=\GR_{\!g}(\ell,p)$ and  $\gr=\gr_{\!g}(\ell, p)$.\\
(\ref{fun1}, \ref{fun2}): Let 
\[
[\A=(A,\lambda_H)], [\A'=(A, \lambda_{H'})]\in\Ver(\GR)
=\Ver(\gr)
\]
with $H, H'\in\SL_g(\cO)$
 positive-definite Hermitian matrices.
By Theorem \ref{g>1}, there exists $M\in\Mat_{g\times g}(\cO)$ with
$M^\dagger HM=\ell^{2n}H'$ for some positive integer $n$.  Hence
by Proposition \ref{degree}, the isogeny $\psi\in\End(A)$ given
by $M$ satisfies $\psi^\ast (\lambda_H)=\ell^{2n} \lambda_{H'}$.
But then by
Theorem \ref{elect} there exists a path of length $2n$ on both $\gr$
and $\GR$ connecting the vertex $[\A']$ to the vertex $[\A]$. \\
\eqref{fun3}: Since $\gr_{\!g}(\ell, p)$ is connected and not bipartite, its
bipartite double cover $\wgr_{\!g}(\ell, p)$ (see Theorem \ref{dealing}
\eqref{dealing1}) is connected.
\end{proof}

\section{The \texorpdfstring{$\ell$}{\unichar{"2113}}-adic uniformization of \texorpdfstring{$\gr_{\!g}(\ell,p)$}{gr\unichar{"005F}g(\unichar{"2113},p)} and
\texorpdfstring{$\wgr_{\!g}(\ell,p)$}{tilde gr\unichar{"005F}g(\unichar{"2113},p)}}
\label{uni}

Through out this section $X$ will be an arbitrary 
principally polarized, not necessarily supersingular, abelian variety.

It is well known that for $\ell\ne p$ the supersingular
elliptic curves over $\bFp$ are in bijective correspondence
with the double cosets
$$\cO_{\H_{p}}[1/\ell]^\times\bs
\GL_2(\Q_\ell)/\Q^\times_\ell\GL_2(\Z_\ell), $$
with
$\GL_2(\Q_\ell)/\Q^\times_\ell\GL_2(\Z_\ell)$ corresponding to the vertices
of the standard tree for $\GL_2(\Q_\ell)$.  We will  generalize
this form to higher dimension, starting with the definition below.

\begin{definition1}
Let $R$ be a commutative ring and $M$ an $R$-algebra with an
anti-involution $x\mapsto x^\dagger$.  We define the {\sf unitary
group} $\rU(M)=\{x\in M \mid x^\dagger x=1\}$.  We define the
{\sf general unitary group} $\GU_R(M)=\{x\in M \mid x^\dagger
x\in R^\times\}$.
\end{definition1}
\begin{remark1}
\label{motel}
  {\rm Let $\cB_{2g}$ be the Bruhat-Tits building for $\GSp_{2g}$ over
    $\Q_\ell$.  The {\sf special $1$-skeleton} $\cSS_{2g}$ of $\cB_{2g}$ 
has vertices the
    special vertices of $\cB_{2g}$ which are the vertices of type
$0$ or $g$ -- see, for example,
\cite[Sect.~2,3]{She}, 
 and edges the edges of the $1$-skeleton
    of $\cB_{2g}$ with both ends special vertices.}
\end{remark1}

Note that the next theorem is true for all principally polarized
abelian varieties whether superspecial or not. Specifically
say that for a principally polarized abelian variety the anti-involution
$x\mapsto x^\dagger$ on $\End(A)$ is Rosati. On $E^g$ we take the
Rosati (anti-)involution corresponding to the product polarization.
Hence on $\Mat_{g\times g}(\OO_E)$ we take $M\mapsto M^\dagger:=\overline{M}^t$,
with $m\mapsto \overline{m}$ the main involution of the definite quaternion
algebra $\End(E)\otimes\Q:=\End^0(E)$.

Theorems similar to Theorem \ref{building q} can be found  in the
theory of Shimura varieties -- see \cite{Ko}, for example.
\begin{theorem}
\label{building q}
Let $X$ be a principally polarized abelian variety of dimension $g$
over an algebraically closed field $k$ of characteristic
$\charr(k)$ with $\ell\neq\charr(k)$ a prime.
The principally polarized abelian varieties
isogenous to $X$ by $\ell$-power isogenies \textup{(}we require that the
principal polarization be the one induced by the isogeny\textup{)} are in
bijective correspondence with the double cosets
\begin{equation}
\label{egg}
\GU(\End(X)[1/\ell])\bs
\GSp_{2g}(\Q_\ell)/\allowbreak \Q^\times_\ell\GSp_{2g}(\Z_\ell) ,
\end{equation}
 with
$\GSp_{2g}(\Q_\ell)/\Q^\times_\ell\GSp_{2g}(\Z_\ell)$ the 
vertices $\Ver(\cSS_{2g})$ as in Remark \textup{\ref{motel}}.  Furthermore,
$(\ell)^g$-isogenies correspond to the edges $\Ed(\cSS_{2g})$.
Specifically, two elements of
\[
\GSp_{2g}(\Q_\ell)/\Q^\times_\ell\GSp_{2g}(\Z_\ell)
\]
are adjacent if
the corresponding homothety classes of unimodular symplectic lattices
have representatives with one having index $(\ell)^g$ in the other.
In particular, the principally polarized superspecial abelian
varieties of dimension $g$ are in bijective correspondence with
\begin{equation}
\label{yolk}
\GU(\MM(\cO_E[1/\ell]))\bs 
\GSp_{2g}(\Q_\ell)/\allowbreak \Q^\times_\ell\GSp_{2g}(\Z_\ell).
\end{equation}
\end{theorem}
\begin{proof}
Let $T=\Ta_\ell(X)$ be the Tate module of $X$, and let
$V=\Ta_\ell(X)\otimes\Q_\ell$, both equipped with the symplectic Weil
pairing.  Identify $\GSp_{2g}(\Q_\ell)=\GSp(V)$ and
$\GSp_{2g}(\Z_\ell)=\GSp(T)$.  Note that we have an exact
sequence $$0\to T\to V\xrightarrow{\pi}X[\ell^\infty]\to0.$$

Let $\phi:X\to X^\prime$ be an $\ell$-power isogeny to an abelian
variety $X^\prime$, principally polarized by the induced polarization.
We will associate to the pair $(\phi,X^\prime)$ the homothety class of
the $\Z_\ell$-lattice $T^\prime=\pi^{-1}(\ker \phi)\subset V$.  Since
the induced polarization on $X^\prime$ is principal, the symplectic
pairing restricted to $T^\prime$ is a scalar multiple of a unimodular
integral pairing.  Conversely, if $[T^\prime]$ is a homothety class of
full-rank $\Z_\ell$-lattices in $V$ such that symplectic pairing restricted to
any representative
is a scalar multiple of a unimodular integral pairing, pick
a representative $T^\prime$ such that $T^\prime\supset T$. Then
$\pi(T^\prime)$ is the kernel of an $\ell$-power isogeny whose image
is principally polarized.  Furthermore, picking a different
representative corresponds to composing the $\phi$ with
multiplication by a scalar power of $\ell$.

Note that if $\psi:X^\prime\to X^{\prime\prime}$ is an
$(\ell)^g$-isogeny, then $T^{\prime\prime}=\pi^{-1}(\ker
\psi\circ\phi)$ is an extension of $T^\prime$ of index $(\ell)^g$.
Hence the corresponding vertices of the building are adjacent.
Conversely, since both graphs have the same degree all special edges
of the building come from $(\ell)^g$-isogenies.

Now let $\Delta$ be the set of all homothety classes of full-rank
$\Z_\ell$-lattices
in $V$ such that the restriction of the symplectic pairing is a scalar
multiple of a unimodular integral pairing.  It is easy to see that
$\Delta\cong\GSp(V)/\Q^\times_\ell\GSp(T)$ with $r\Q^\times_\ell\GSp(T)$
corresponding to the class $[rT]$.

It suffices to show that $[rT]$ and $[sT]$ correspond to
isomorphic principally polarized abelian varieties if and only if
$[rT]=[\psi sT]$ for some
$\psi\in\GU(\End(X)[1/\ell])$.
After possibly scaling $r$, $s$, and $\psi$ by powers of $\ell$ we may assume
that $\psi\in\End(X)$, $rT=\psi sT$, and $rT,\,sT\supset T$.
Therefore $\pi(rT)=\psi(\pi(sT))$.  Hence if $\ker \phi=\pi(rT)$ and
$\ker \phi^\prime=\pi(sT)$, then $\phi \circ \psi = \phi^\prime$ and
both have the same codomain.

Conversely, if $\phi$ and $\phi^\prime$ have the same codomain, let
$\psi=\widehat{\phi}\circ \phi^\prime$.  Note that
$\psi\in\GU(\End(X)[1/\ell])$
since it preserves the polarization.  Now $\phi \circ \psi =
\deg(\phi) \phi^\prime$.  Now let $\pi(rT)=\ker \phi$ and
$\pi(sT)=\ker(\deg(\phi) \phi^\prime)$.  Then $rT=\psi sT$, and we are
done with the main claim.

The final assertion with  \eqref{yolk} now follows from Theorem
\ref{fun}.
\end{proof}
We now apply Theorem \ref{building q} to derive the $\ell$-adic 
uniformization of the isogeny graphs $\gr_{\!g}(\ell, p)$ and
$\wgr_{\!g}(\ell, p)$.
\subsection{The case \texorpdfstring{\except{toc}{\boldmath{$g=1$}}\for{toc}{$g=1$}}{g=1}: \texorpdfstring{\except{toc}{\boldmath{$\ell$}}\for{toc}{$\ell$}}{\unichar{"2113}}-adically uniformizing 
\texorpdfstring{\except{toc}{\boldmath{$\gr_1(\ell,p)$}}\for{toc}{$\gr_1(\ell,p)$}}{gr\unichar{"005F}1(\unichar{"2113},p)} and \texorpdfstring{\except{toc}{\boldmath{$\wgr_1(\ell,p)$}}\for{toc}{$\wgr_1(\ell,p)$}}{tilde gr\unichar{"005F}1(\unichar{"2113},p)}}
\label{g=1q}
Let $\Delta=\Delta_\ell$ be the tree for $\SL_2(\Q_\ell)$. The
rational definite quaternion algebra $\HH_p$ with maximal
order $\cO=\cO_{\HH_p}$ is ramified at $p$
and split at $\ell$.  Set 
\begin{equation}
\label{gamma}
\Gamma_0=\cO[1/\ell]^\times\quad\text{and}\quad
\Gamma_1=\{\gamma\in\Gamma_0\mid \Norm_{\HH_p/\Q}(\gamma)=1\}.
\end{equation}
We have $\Gamma_0=\cO[1/\ell]^\times\hookrightarrow 
(\HH_p\otimes_{\Q}\Q_{\ell})^\times=\GL_2(\Q_\ell)$ and likewise
$\Gamma_1\hookrightarrow \GL_2(\Q_\ell)$. Let $\overline{\Gamma}_i$
be the image of $\Gamma_i$ in $\PGL_2(\Q_\ell)$ for $i=0,1$.
The groups $\overline{\Gamma}_0$, $\overline{\Gamma}_1$
are discrete cocompact subgroups of $\PGL_2(\Q_\ell)$.
The groups $\Gamma_i\subset \GL_2(\Q_\ell)$ act on $\Delta$ through
their image $\overline{\Gamma}_i\subseteq \PGL_2(\Q_\ell)$, $i = 0,1$.
Hence the quotients $\ggr_1\colonequals
\Gamma_1\backslash \Delta=
\overline{\Gamma}_1\backslash\Delta$
and $\ggr_0\colonequals \Gamma_0\backslash \Delta=
\overline{\Gamma}_0\backslash \Delta$ are finite graphs
with weights. Kurihara \cite[p.~294]{Kur} shows that the weighted
adjacency matrix $\Adw(\ggr_0)$ is the Brandt matrix $B_1(\ell)$ for
$\cO\subseteq\HH_p$; we know $\Adw(\gr_1(\ell, p))= B_1(\ell)$
by Theorem \ref{smad}. In fact, to  show $\Adw(\ggr_0)=B_1(\ell)$
Kurihara basically shows $\ggr_0=\br_1(\ell, p)$. In \cite[p.~296]{Kur}
it is shown that $\ggr_1$ (note that our $\overline{\Gamma}_1$ is $\Gamma_+$ in
\cite{Kur}) is the bipartite double cover of $\ggr_0$.  Hence we
have
\begin{theorem}
\label{oneee}
{\rm (Kurihara) } 
\begin{enumerate}[\upshape (a)]
\item
\label{oneee1}
$\br_1(\ell, p)
=\Gamma_0\backslash \Delta_\ell=
\overline{\Gamma}_0\backslash \Delta_\ell$ as graphs with weights.
\item
\label{oneee2}
$\Gamma_1\backslash \Delta_\ell=\overline{\Gamma}_1\backslash \Delta_\ell$ is the bipartite double cover of 
$\Gamma_0\backslash \Delta_\ell =\overline{\Gamma}_0\backslash \Delta_\ell$.
\end{enumerate}
\end{theorem}
\begin{theorem}
\label{twooo}
\begin{enumerate}[\upshape (a)]
\item
\label{twooo1}
$\gr_1(\ell,p)=\Gamma_0\backslash \Delta_\ell=
\overline{\Gamma}_0\backslash \Delta_\ell$ as graphs with weights.
\item
\label{twooo2}
$\wgr_1(\ell,p)=\Gamma_1\backslash\Delta_\ell=
\overline{\Gamma}_1\backslash\Delta_\ell$ as graphs with weights.
\end{enumerate}
\end{theorem}
\begin{proof}
\eqref{twooo1}: Combine Theorem \ref{smad}\eqref{smad1}
with Theorem \ref{oneee}\eqref{oneee1}.\\
\eqref{twooo2}: Combine Theorem \ref{oneee}\eqref{oneee2}
with Theorem \ref{dealing}\eqref{dealing1}.
\end{proof}
\begin{remark1}
\label{fish}
{\rm  
\begin{enumerate}[\upshape (a)]
\item
\label{fish1}
The big isogeny graph $\GR_1(\ell, p)$ is {\em not}
$\ell$-adically uniformized by $\Delta_\ell$ since $\GR_1(\ell,p)$ is not
a graph with opposites.
\item
\label{fish2}
Theorem \ref{twooo}\eqref{twooo1} obviously implies that the isogeny
graph $\gr_1(\ell, p)$ is connected.  Note that in fact Kurihara
\cite[p.~291]{Kur} invokes strong approximation in the course of 
proving $\br_1(\ell,p)=\overline{\Gamma}_0\backslash \Delta_\ell$.
\end{enumerate}
}
\end{remark1}
\subsection{The isogeny graphs \texorpdfstring{\except{toc}{\boldmath{$\gr_1(\ell,p)$}}\for{toc}{$\gr_1(\ell,p)$}}{gr\unichar{"005F}1(\unichar{"2113},p)}, 
\texorpdfstring{\except{toc}{\boldmath{$\wgr_1(\ell,p)$}}\for{toc}{$\wgr_1(\ell,p)$}}{tilde gr\unichar{"005F}1(\unichar{"2113},p)} and
Shimura curves}
\label{eel}
Theorem \ref{twooo} in turn will show that our isogeny graphs $\gr_{\!1}(\ell,p)$
and $\wgr_{\!1}(\ell,p)$ arise from the bad reduction of Shimura curves,
which we now explain. Let 
$B$ be the indefinite rational quaternion division algebra with
$\Disc B=\ell p$. Let 
$V_B/\Q$ be the Shimura curve parametrizing principally polarized abelian
surfaces with QM (quaternionic multiplication)
 by a maximal order $\mathcal{M}\subseteq B$.  There
is then a model $M_B/\Z$ of $V_B/\Q$ constructed as a coarse  moduli
scheme by Drinfeld \cite{drin}; see also \cite{jl1}. Let 
$\msL/\Z_\ell$ be the $\ell$-adic upper half-plane.
The dual graph $G(\msL/\Z_\ell)$
of its
special fiber is canonically $\Delta=\Delta_\ell$.  For
$\overline{\Gamma}\subseteq
\PGL_2(\Q_\ell)$ a discrete, cocompact subgroup,
the quotient $\overline{\Gamma}\backslash\msL$
is the formal completion of a scheme $\msL_{\overline{\Gamma}}/\Z_\ell$
along its closed fiber. The dual graph of its special fiber
$G(\msL_{\overline{\Gamma}}/\Z_\ell)
\simeq (\overline{\Gamma}\backslash\Delta)^\ast$ as graphs with lengths in the notation
of Definition \ref{weight1}\eqref{wad1}, see \cite[Prop.~3.2]{Kur}.

For the formulation below, see \cite[Theorems 4.3$^\prime$, 4.4]{jl1}.
\begin{theorem}
\textup{(\v{C}erednik, Drinfeld) }
\label{font}
Let $w_\ell$ be the Atkin-Lehner involution at $\ell$ of $M_B$.
Let $\overline{\Gamma}_0$ be the image of $\Gamma_0\subseteq \GL_2(\Q_\ell)$
in $\PGL_2(\Q_\ell)$ and similarly for $\overline{\Gamma}_1$.
Let $\mathfrak{O}$ be the ring of integers in the unramified quadratic
extension of $\Q_\ell$.
\begin{enumerate}[\upshape (a)]
\item
The scheme $M_B\times\Z_{\ell}$ is the twist of $\msL_{\overline{\Gamma}_1}/\Z_\ell$
given by the $1$-cocycle 
\begin{align*}
\chi\in H^1(\Gal(\mathfrak{O}/\Z_\ell)&,
\Aut(\msL_{\overline{\Gamma}_1}\times_{\Z_\ell}\mathfrak{O}/\mathfrak{O})),
\text{ where } \chi:\Frobb_\ell\mapsto w_\ell:\\
&\quad M_B\times \Z_\ell=(\msL_{\overline{\Gamma}_1})^\chi .
\end{align*}
\item
$(M_B/w_\ell)\times\Z_{\ell}=\msL_{\overline{\Gamma}_0}/\Z_\ell$.
\end{enumerate}
\end{theorem}
The curve $M_B\times\Z_{\ell}/\Z_\ell$ is an {\em{admissible curve}}
in the sense of \cite[Defn.~3.1]{jl1}.  As such, the dual
graph of its special fiber $\GG(M_B\times\Z_{\ell}/\Z_\ell)$ is a graph
with lengths as in Definition \ref{weight1}\eqref{weight12} by 
\cite[Defn.~3.2]{jl1}.
\begin{corollary}
\label{font1}
\begin{enumerate}[\upshape (a)]
\item
\label{font11}
$G(M_B\times \Z_\ell/\Z_\ell)=\overline{\Gamma}_1\backslash \Delta=
\wgr_{\!1}(\ell,p)$
as graphs with lengths.
\item
$G((M_B/w_\ell)\times\Z_\ell/\Z_{\ell})=(\overline{\Gamma}_{0}\backslash\Delta)^\ast=
\gr_{\!1}(\ell,p)^\ast$ as graphs with lengths
with $(\overline{\Gamma}_{0}\backslash\Delta)^\ast$,
$\gr_{\!1}(\ell,p)^\ast$ as in Definition \textup{\ref{weight1}\eqref{wad1}}.
\end{enumerate}
\end{corollary}
\begin{proof}
This follows from Theorem \ref{font} by \cite[Prop.~4.2]{jl1},
which in turn is extracted from \cite[\S 3]{Kur}.

\end{proof}

\subsection{The general case \texorpdfstring{\except{toc}{\boldmath{$g\geq1$}}\for{toc}{$g\geq1$}}{g\unichar{"2265}1}: 
\texorpdfstring{\except{toc}{\boldmath{$\ell$}}\for{toc}{$\ell$}}{\unichar{"2113}}-adically uniformizing
\texorpdfstring{\except{toc}{\boldmath{$\gr_{\!g}(\ell,p)$}}\for{toc}{$\gr_{\!g}(\ell,p)$}}{gr\unichar{"005F}g(\unichar{"2113},p)} and \texorpdfstring{\except{toc}{\boldmath{$\wgr_{\!g}(\ell,p)$}}\for{toc}{$\wgr_{\!g}(\ell,p)$}}{tilde gr\unichar{"005F}g(\unichar{"2113},p)}}
\label{twoo}
Recall $A=E^g$, $\cO=\End(E)\subseteq \HH_p$, and 
$\End(A)=\Mat_{g\times g}(\cO)$.
Let $\cB_{2g}$ be the Bruhat-Tits building for $\Sp_{2g}(\Q_\ell)$
and $\cSS_{2g}$ its special $1$-skeleton as in Remark \ref{motel}.
Note that \mbox{$\GU_g(\HH_p\otimes_{\Q}\Q_\ell)$} and $\rU_g(\HH_p\otimes_{\Q}
\Q_\ell)$ as in \eqref{wet} act on $\cSS_{2g}$ with finite quotient.
\begin{theorem}
\label{Un}
\begin{enumerate}[\upshape (a)]
\item
\label{Un1}
$\br_{\!g}(\ell, \cO)=\GU_g(\cO[1/\ell])\backslash \cSS_{2g}$ as graphs
with weights.
\item
\label{Un2}
$\rU_g(\cO[1/\ell])\backslash \cSS_{2g}$  is the bipartite double cover
of $\GU_g(\cO[1/\ell])\backslash \cSS_{2g}$.
\end{enumerate}
\end{theorem}
\begin{proof}
\eqref{Un1} follows immediately from Theorem \ref{building q} since $\GU_g(\cO[1/\ell])\backslash \cSS_{2g}$ is the same as \eqref{yolk}.

\eqref{Un2} follows from (1) since $\wgr_g$ is the bipartite double cover of $\gr_g$ and $\PU_g(\cO[1/\ell])$ is the subgroup of $\PGU_g(\cO[1/\ell])$ that preserves mod $2$ distance.
\end{proof}
\begin{theorem}
\label{Dos}
\begin{enumerate}[\upshape (a)]
\item
\label{Dos1}
$\gr_{\!g}(\ell, p)=\GU_g(\cO[1/\ell])\backslash \cSS_{2g}$ 
as graphs with weights.
\item
\label{Dos2}
$\wgr_{\!g}(\ell, p)=\rU_g(\cO[1/\ell])\backslash \cSS_{2g}$ 
as graphs with weights.
\end{enumerate}
\end{theorem}
\begin{proof}
\eqref{Dos1}:
Combine Theorem \ref{smad}\eqref{smad1}
with Theorem \ref{Un}\eqref{Un1}.\\
\eqref{Dos2}: Combine Theorem \ref{Un}\eqref{Un2}
with Theorem \ref{dealing}\eqref{dealing1}.
\end{proof}
\begin{remark1}
\label{chew}
{\rm 
\begin{enumerate}[\upshape (a)]
\item
\label{chew1}
Theorem \ref{Dos} once again immediately implies that the 
isogeny graphs $\gr_{\!g}(\ell, p)$ and $\wgr_{\!g}(\ell, p)$ are connected.
However, note that the proof of Theorem \ref{Dos} uses Theorem
\ref{building q}, which in turn uses Theorem \ref{fun}.

\item
\label{chew2}
In case $g=1$ we have $\Sp_2(\Q_\ell)=\SL_2(\Q_\ell)$,
$\cSS_2=\Delta_\ell$, $\rU_1(\cO[1/\ell])=\Gamma_1$, and
$\GU_1(\cO[1/\ell])=\Gamma_0$.  Hence for $g=1$ we recover
Theorem \ref{twooo}.
\item
\label{chew3}
The big isogeny graph $\GR_{\!g}(\ell,p)$ is {\em not}
uniformized by $\cSS_{2g}$ as in the $g=1$ case (Remark \ref{fish})
since it is not a graph with opposites.
\item
\label{chew4}
There would be great interest in generalizing Theorem \ref{font}
and Corollary \ref{font1} to $g>1$.
\end{enumerate}
}
\end{remark1}

\section{Computations: The Ramanujan property for  \texorpdfstring{$\GR_{\!g}(\ell,p)$}{Gr\unichar{"005F}g(\unichar{"2113},p)} with \texorpdfstring{$g>1$}{g>1}}
\label{Ram}
\subsection{A non-Ramanujan example}
\label{Ram1}
To see that  the isogeny graph $\GR_{\!g}(\ell,p)$
is in general non-Ramanujan, consider the case
$\ell=2$, $g=2$, and $p=11$.  Here there are two 
supersingular elliptic curves: $E_1:y^2 = x^3 + 1$ and $E_2:y^2 = x^3
+ x$.  There are also two superspecial genus-$2$ curves: $C_1: y^2=x^6
+ 1$ and $C_2: y^2 = x^6 + 3 x^3 + 1$.  Hence there are five principally
polarized
superspecial abelian surfaces: $E_1\times E_1$, $E_2\times E_2$,
$E_1\times E_2$, and the jacobians $J(C_1)$, $J(C_2)$, the products taken with the product
polarization and the jacobians with their canonical polarizations.  A Richelot isogeny of a principally polarized abelian surface
$(A, \lambda)$ is quotienting by a maximal isotropic subgroup
of $A[2]$.  We computed the Richelot
isogenies  for these principally polarized
abelian surfaces using Magma~\cite{BCP}.  The
adjacency matrix for $\GR_{2}(2,11)$ is
$$
\Ad(\GR_2(2,11))=
\begin{bmatrix}
3 & 9 & 0 & 3 & 0 \\
4 & 3 & 4 & 4 & 0 \\
0 & 3 & 6 & 0 & 6 \\
1 & 3 & 0 & 3 & 8 \\
0 & 0 & 3 & 4 & 8
\end{bmatrix};
$$
the row-sums of this matrix are all $15=N_2(2)=(1+2)(1+2^2)$.
The eigenvalues of this matrix are $15$, $7\pm\sqrt{3}$, and
$-3\pm\sqrt{3}$.  The second largest of these is $7+\sqrt{3} > 2\sqrt{14}$.
 Hence the graph is not Ramanujan.
\subsection{A Ramanujan example}
\label{Ram2}
To see that the isogeny graph $\GR_{\!g}(\ell,p)$ can (rarely) be
Ramanujan, consider the case $\ell=2$, $g=2$, and $p=7$. In characteristic
$7$ there
is one supersingular elliptic curve $E: y^2=x^3-x$ and one superspecial
genus-$2$ curve $C:y^2=x^5+x$. There are two principally polarized superspecial
abelian surfaces: $E\times E$ with the product polarization and the 
jacobian $J(C)$ of $C$ with its canonical polarization.
The adjacency matrix for $\GR_2(2,7)$ is
$$
\Ad(\GR_2(2,7))=\begin{bmatrix}
11 & 4\\
6 & 9\end{bmatrix}.
$$
Again, the graph $\GR_2(2,7)$ is $15$-regular and
we see that the row sums of $\Ad(\GR_2(2,7))$
are all $15$.
The eigenvalues of this matrix are $15$ and $5$.
Since $5<2\sqrt{14}$, the graph $\GR_2(2,7)$ is Ramanujan.

\subsection{A range of computations}
\label{Ram3}
We computed $\GR_{\!g}(\ell,p)$ using Theorem \ref{adbig} by
calculating the Brandt matrix $B_g(\ell)$ for the maximal order
$\cO=\End(E)\subseteq \HH_p$.  We were able to do this
for all primes $p\le p_{\rm max}$ and
$(g,\ell,p_{\rm max})$ one of $(2,2,311)$, $(2,3,257)$, $(2,5,173)$,
$(3,2,41)$, $(3,3,23)$.  Hence in these ranges we could determine
whether the big isogeny graph $\GR_g(\ell,p)$ is Ramanujan.

The graph is trivially Ramanujan, due to having only one vertex, when
$(g,p)=(2,2),\, (2,3),$ or $(3,2)$ and $\ell$ arbitrary -- the number of
vertices only depends on $(g,p)$ and not on $\ell$.

Otherwise, the only Ramanujan examples we found are when $(g,\ell,p)$
is one of $(2,2,5)$, $(2,2,7)$, $(2,3,7)$, $(3,2,3)$.  (All these
graphs have two vertices, but not every two-vertex graph is
Ramanujan.)

\section*{Acknowledgments}
The proof of Theorem \ref{hen}\eqref{hen7} using Gelfand's trick
 was worked out 
with Guy Henniart and Marie-France Vign\'{e}ras.  
Nicholas Shepherd-Barron suggested 
using \cite[Chap.~11]{vM} to prove our Theorem  \ref{yam}.  
Ching-Li Chai's comments
greatly improved our paper.
We also benefited from conversations and correspondence with
Tomoyoshi Ibukiyama and Frans Oort. The referee's careful reading
and suggestions went beyond the call of duty.
It is a pleasure to
thank them all.

\bibliographystyle{plain}
\bibliography{IG9b}

\end{document}